\documentclass[12pt, oneside, notitlepage]{amsart}
\usepackage[margin=3cm]{geometry}
\usepackage{amsmath,amssymb,amsthm,graphicx,mathrsfs,bbm,url}
\usepackage{amsthm}
\usepackage{wrapfig}
\usepackage{enumitem}
\usepackage{mathtools}
\usepackage[utf8]{inputenc}
 \usepackage[T1]{fontenc}
\usepackage[usenames,dvipsnames]{color}
\usepackage[colorlinks=true,linkcolor=Red,citecolor=Green]{hyperref}
\usepackage[super]{nth}
\usepackage[open, openlevel=2, depth=3, atend]{bookmark}
\hypersetup{pdfstartview=XYZ}
\usepackage[font=footnotesize]{caption}

\usepackage{graphicx}
\usepackage{epstopdf}
 
\usepackage{hyperref}

\theoremstyle{plain}
\newtheorem{theorem}{Theorem}[section]
\newtheorem*{theorem*}{Theorem}

\newtheorem{lemma}[theorem]{Lemma}
\newtheorem{proposition}[theorem]{Proposition}

\theoremstyle{definition}
\newtheorem{definition}[theorem]{Definition}

\theoremstyle{remark}
\newtheorem{remark}[theorem]{Remark}

\numberwithin{equation}{section}

\newcommand{\C}{\mathbb{C}}
\newcommand{\R}{\mathbb{R}}

\newcommand{\Z}{\mathbb{Z}}
\newcommand{\D}{\mathbb{D}}

\newcommand{\T}{\mathbb{T}}

\newcommand{\HH}{\mathbb{H}}

\newcommand{\eps}{\varepsilon}

\newcommand{\mc}{\mathcal}

\newcommand{\dd}{\mathrm{d}}

\DeclareMathOperator{\vol}{vol}

\DeclareMathOperator{\Op}{Op}
\DeclareMathOperator{\WF}{WF}

\DeclareMathOperator{\supp}{supp}

\DeclareMathOperator{\comp}{comp}

\newcommand{\be}{\begin{equation}}
\newcommand{\ee}{\end{equation}}

% ...

% ...

\title
[Zonal states and improved $L^\infty$ bounds]
{Zonal states and improved $L^\infty$ bounds for eigenfunctions of magnetic Laplacians on hyperbolic surfaces}

\author{Ambre Chabert (*)}
\thanks{(*) Corresponding author. Email address: ambre.chabert@ens.fr}
\address{Département de Mathématiques et Applications, Ecole Normale Supérieure, UMR 8553, 45 rue d’Ulm. 75230 Paris Cedex 05, France}
\email{ambre.chabert@ens.fr}

\author{Thibault Lefeuvre}
\address{Université Paris-Saclay, CNRS, Laboratoire de mathématiques d’Orsay, 91405, Orsay, France}
\email{thibault.lefeuvre1@universite-paris-saclay.fr}

\begin{document}

\begin{abstract}
We establish polynomially improved $L^\infty$ bounds for eigenfunctions of magnetic Laplacians on hyperbolic surfaces in the critical energy regime. We also show that, below the critical energy, the Hörmander bound is saturated by explicit eigenstates, which we call \emph{magnetic zonal states}. These states resemble zonal harmonics on the sphere and equidistribute on Lagrangian tori in phase space.
\end{abstract}

\maketitle

\section{Introduction}

\subsection{Setup} Let $(\Sigma, g)$ be a closed oriented hyperbolic surface of genus $\textsl{g} \geq 2$. Let $L \to \Sigma$ be a complex Hermitian line bundle equipped with a unitary connection $\nabla$. We denote its curvature by $F_\nabla = -i B \vol$, where $B \in C^\infty(\Sigma)$ and $\vol$ is the Riemannian volume of the hyperbolic metric $g$. The function $B$ is called the \emph{magnetic field}. In the following, we will always assume that the magnetic field is \emph{constant}. Notice that, by the Gauss-Bonnet theorem, $2B(\textsl{g} -1) \in \Z$ is equal to the first Chern class of $L$. Up to replacing $(L,\nabla)$ by its dual, we also assume throughout that $B>0$.

The connection $\nabla$ induces a connection $\nabla^{\otimes k}$ on tensor powers $L^{\otimes k} \to \Sigma$ of $L$, where $k \in \Z$, and for $k \leq 0$, $L^{\otimes k} := (L^*)^{\otimes |k|}$, with curvature $F_{\nabla^{\otimes k}} = - i k B \vol$. The \emph{magnetic Laplacian} is then defined as the $k$-dependent family of operators
\[
\Delta_k := \tfrac{1}{2} (\nabla^{\otimes k})^* \nabla^{\otimes k} : C^\infty(\Sigma, L^{\otimes k}) \to C^\infty(\Sigma,L^{\otimes k}).
\]

The first $N_k := \lfloor kB \rfloor$ eigenvalues of $\Delta_k$ are explicit and given by
\begin{equation}
\label{equation:lambdamk}
\lambda_{k,m} := kB(m+\tfrac{1}{2}) - \dfrac{m(m+1)}{2}, \qquad m \in \{0, ..., N_k-1\},
\end{equation}
see \cite[Proposition 2.8]{Charles-Lefeuvre-25}. We are interested in the high-frequency limit of eigenfunctions
\begin{equation}
\label{equation:eigenmodes}
\Delta_k u_k = \mu_k u_k, \qquad \mu_k = k^{2}(E+o(1)), \qquad \|u_k\|_{L^2(\Sigma,L^{\otimes k})}=1,
\end{equation}
as $k \to +\infty$, where $E \geq 0$ is a fixed energy. {\color{black}This corresponds to the semiclassical regime where the magnetic field $kB$ tends to infinity}; the semiclassical parameter is $h = 1/k$.

The defect measures associated with the sequence $(u_k)_{k \geq 0}$ are supported on the energy shell $\{p=E\}$, where $p(x,\xi) := \tfrac{1}{2}|\xi|^2_g$ is the principal symbol of $k^{-2}\Delta_k$. In addition, these measures are also invariant by the \emph{magnetic flow} $(\Phi_t)_{t \in \R}$, which is the Hamiltonian flow on $T^*\Sigma$ generated by the principal symbol $p$ and calculated with respect to the twisted symplectic form
\begin{equation}
    \label{equation:twisted-form}
\Omega := \omega_0 + i\pi^*F_{\nabla},
\end{equation}
where $\omega_0$ is the Liouville $2$-form on $T^*\Sigma$ and $\pi : T^*\Sigma \to \Sigma$ is the footpoint projection, see \cite{Charles-Lefeuvre-25}.

We let $E_c := \tfrac{1}{2}B^2$. Note that for $m = N_k = \lfloor kB \rfloor$, one has $k^{-2} \lambda_{k,N_k} = E_c + \mc{O}(k^{-1})$. According to the value of $E$, the magnetic flow has a very different behaviour in restriction to the energy shells $\{p=E\}$, see \cite[Proposition 2.5]{Charles-Lefeuvre-25}:
\begin{enumerate}[label=(\roman*)]
\item For $0 < E < E_c$, all the orbits of the magnetic flow are \emph{periodic} of period $2\pi (B^2-2E)^{-1/2}$;
\item For $E=E_c$, the magnetic flow is \emph{uniquely ergodic}, that is all trajectories are dense and equidistribute towards the unique smooth flow-invariant measure, called the Liouville measure;
\item For $E > E_c$, the flow is \emph{uniformly hyperbolic} (also called \emph{Anosov} in the literature).
\end{enumerate}

This has deep consequences and constrains the possible measures that may arise as defect measures of the sequence \eqref{equation:eigenmodes}. In particular, it was established in \cite[Theorem 1.2]{Charles-Lefeuvre-25}, that if the sequence $(u_k)_{k \geq 0}$ satisfies the stronger assumption
\begin{equation}
\label{equation:eigenmodes2}
\Delta_k u_k = \mu_k u_k, \qquad \mu_k = k^{2}(E_c+\mc{O}(k^{-\ell})), \qquad \|u_k\|_{L^2(\Sigma,L^{\otimes k})}=1,
\end{equation}
for some real number $\ell > 0$, then there exists $C > 0$ such that for all $a \in C^\infty(\Sigma)$:
\begin{equation}
\label{equation:concentration}
\left|\langle a u_k, u_k \rangle_{L^2(\Sigma,L^{\otimes k})}- \dfrac{1}{\vol(\Sigma)}\int_{\Sigma} a(x) \vol(x)\right| \leq C k^{-\theta \min(\ell,1/15)/4100}\|a\|_{C^{17}(\Sigma)}.
\end{equation}

(A more general estimate was actually established in \cite{Charles-Lefeuvre-25} which holds for observables $a \in C^\infty_{\mathrm{comp}}(T^*\Sigma)$ compactly supported in phase space.) Here $0 < \theta < 1/2$ is any real number satisfying $\theta(1-\theta) \leq \lambda_1(\Sigma)$, the first non-zero eigenvalue of the Laplacian acting on functions.

\subsection{Subcritical energy regime: magnetic zonal states}

The first purpose of this article is to illustrate how the behavior of the magnetic flow on each energy shell $\{p = E\}$ relates to the $L^\infty$ norms of eigenfunctions satisfying \eqref{equation:eigenmodes}. We also discuss the $L^p$ norms of eigenstates in \S\ref{ssection:lp} below. We start with the \textit{subcritical regime} where the magnetic flow is periodic on each energy shell.

\subsubsection{Statement} The Hörmander bound \cite{Hormander-68} asserts that for general elliptic pseudodifferential operators, the $L^\infty$ norm can be bounded by $\|u_k\|_{L^\infty} \leq C k^{(n-1)/2}\|u_k\|_{L^2}$ for some uniform constant $C > 0$, where $n$ is the dimension of the ambient manifold (see also \cite{Levitan-52,Avakumovic-56}). Here $n=\dim(\Sigma) =2$, so the Hörmander bound reads $\|u_k\|_{L^\infty} \leq C k^{1/2}\|u_k\|_{L^2}$ in our context. Let us note that Hörmander's bound applies to the magnetic Laplacian since this bound is microlocal, so it applies to twisted pseudodifferential operators as they are (locally) microlocally conjugated to standard pseudodifferential operators (see \S\ref{ssection:twisted-pseudos}).

{\color{black}Our first theorem is that we can construct a sequence of eigenfunctions saturating the Hörmander bound:

\begin{theorem}[Saturation of $L^\infty$ norm]\label{theorem:magneticzonalstates}
    Let $0 \leq E < E_c$. Then, there exists a sequence $(u_k)_{k\geq 0}$ satisfying \eqref{equation:eigenmodes} such that:
    \begin{equation}
    \label{equation:saturation}
\liminf_{k \to +\infty} k^{-1/2}\|u_k\|_{L^\infty(\Sigma,L^{\otimes k})} > 0.
\end{equation}
\end{theorem}}

The eigenfunctions that we construct bear a strong resemblance to the zonal harmonics on the sphere. Indeed, they can be constructed by closely analogous methods. Hence, by analogy with the spherical setting, we shall refer to these eigenfunctions as \textit{magnetic zonal states}. {\color{black}Improvements on $L^\infty$ bounds are usually connected to improvements on the remainder in the Weyl law. However, in our setting, the eigenvalues below the critical energy are completely explicit and given by \eqref{equation:lambdamk}.} {\color{black}  The eigenvalues $\lambda_{k,m}$ become less and less spaced as $m$ tends to $Nk$ (see \eqref{equation:lambdamk}), and their multiplicity decreases as well (see \cite[Proposition 2.8]{Charles-Lefeuvre-25}). Equivalently, the spectrum becomes more distributed when approaching the critical energy $E_c$ from below, and the remainder in the semiclassical Weyl law is indeed a sharp $\mc{O}(k)$ at a fixed $E < E_c$, as for the round sphere, as opposed to the polynomial improvement that one can establish in the remainder at $E_c$. This is consistent with the (im)possibility of improving Hörmander’s $L^\infty$ bound, depending on the energy level.}

{\color{black}It has been observed in the literature that maximal $L^\infty$ growth of eigenfunctions is strongly linked to recurence or periodicity of the underlying classical flow, see \cite{sogge2002riemannian,sogge2011blowup}. Here, the magnetic flow is \textit{periodic} on each subcritical energy shell, which should be compared to the periodicity of the geodesic flow on the sphere. Periodicity has been exploited in many previous works to build \textit{quasimodes} saturating Hörmander's bound for Zoll manifolds, {\color{black} or, more generally, when the bichacteristics starting a a fixed point in cotangent space are all periodic,} using the formalism of \textit{oscillatory functions} on \textit{Lagrangian submanifolds} (see \cite{duistermaat1974oscillatory}). {\color{black} Our construction is to be compared with existing results in the presence of periodicity, such as \cite{ZELDITCH1997415,zelditch1977wave, Galkowski-Toth-18, AIF_2019__69_4_1757_0}. We emphasize that it is different from the one in the above-cited references, as the authors typically rely on exhibiting an invariant Lagrangian torus which satisfies the Bohr-Sommerfeld quantization conditions, with a maximally degenerate caustic at a given point, and then building quasimodes associated to this Lagrangian submanifold through the standard formalism of oscillatory integrals. This method makes it usually hard to obtain quasimodes of order $\infty$, thus eigenfunctions, because of geometric obstructions (see \cite{zelditch2014gaussianbeamszollmanifolds}). On the other hand, we first build Gaussian beams eigenfunctions, and then build zonal modes through an averaging procedure.}}

\subsubsection{Construction of the zonal states} The magnetic zonal states with energy $0 < E < E_c$ that saturate Hörmander's bound in \eqref{equation:saturation} are explicitly constructed as follows. Given a point $x_0 \in \Sigma$, the energy shell
\[
C(x_0,E) := \{p=E\} \cap T^*_{x_0}\Sigma
\]
is diffeomorphic to a circle (for $E=0$, $C$ is reduced to a point and one needs to argue differently). As the magnetic flow $(\Phi_t)_{t \in \R}$ is $T_E$-periodic on $\{p=E\}$ with $T_E := 2\pi(B^2-2E)^{-1/2}$, for any $\xi \in C(x_0,E)$, $\Phi_{T_E}(x_0,\xi) = (x_0,\xi)$; additionally, the loop $t \mapsto \Phi_t(x_0,\xi), t \in [0,T_E]$ projects via the footpoint projection $\pi : T^*\Sigma \to \Sigma$ to a smooth immersed loop $\gamma : [0, T_E] \to \Sigma$. This trajectory is called a \emph{magnetic geodesic}.

Given $\xi \in C(x_0,E)$, a Gaussian state $\mathbf{e}_{k,x_0,\xi} \in C^\infty(\Sigma,L^{\otimes k})$ localized at $(x_0,\xi)$ can be used to produce a genuine eigenfunction 
\[
\mathbf{f}_{k,x_0,\xi} := k^{1/4} \Pi_E \mathbf{e}_{k,x_0,\xi} \in C^\infty(\Sigma,L^{\otimes k})
\]
satisfying \eqref{equation:eigenmodes} at energy $E$, by applying a suitable orthogonal projector $\Pi_E$, see \cite[Proposition 4.6]{Charles-Lefeuvre-25} and \S\ref{section:proof} below. The eigenstate $\mathbf{f}_{k,x_0,\xi}$ is a (magnetic) Gaussian beam which is quasi $L^2$ normalized; it is microlocally concentrated on the periodic bicharacteristic generated by $(x_0,\xi)$ in $T^*\Sigma$, {\color{black}with width $k^{-1/2}$, and its norm at $p_0$ is $k^{1/4}$.} One can then define the magnetic zonal state by averaging over all Gaussian beams based at $C(x_0,E)$, that is
\begin{equation}
    \label{equation:magnetic-zonal-state}
    \mathbf{u}_{k,x_0} := k^{1/4} \int_{0}^{2\pi} {\color{black}\chi(\theta)}\mathbf{f}_{k,x_0,R_\theta\xi_0} \dd \theta,
\end{equation}
where $R_\theta : T^*\Sigma \to T^*\Sigma$ is the fiberwise rotation by angle $\theta \in [0,2\pi]$ induced by the orientation, $\xi_0 \in C(x_0,E)$ is arbitrary, {\color{black}and $\chi \in C^\infty(C(x_0,E))$ is a bump function with small support around a point. The construction is done in such a way that, at first, only smooth functions with support in a small open subset of $C(x_0,E)$ are allowed, see \eqref{equation:calor} for a precise definition. This issue can easily be avoided, but such a refinement is not even necessary to construct eigenstates saturating Hörmander's bound.} We do not introduce a formal definition for magnetic zonal states: by this, we simply mean eigenfunctions of the form \eqref{equation:magnetic-zonal-state} saturating Hörmander's $L^\infty$ bound.

{\color{black}The heuristic for the normalizing factor $k^{1/4}$ is that, since the Gaussian beams have width $k^{-1/2}$, one can cover $C(x_0,E)$ by roughly $k^{1/2}$ almost orthogonal Gaussian beams, thus yielding a $k^{-1/4}$ factor to have the $L^2$ normalization. The measure of integration should thus be $k^{-1/2}\dd \theta$ to account for the width of Gaussian beams, overall yielding a factor $k^{1/4}$. In particular, the norm of $\mathbf{u}_{k,x_0}$ at $p_0$ is of size $k^{1/4} \times k^{1/4} = k^{1/2}$ so it has maximal $L^{\infty}$ norm growth. We emphasize that one also needs to make sure all the $\mathbf{f}_{k,x_0,R_\theta\xi_0}(p_0)$ have the same phase in order to avoid destructive interferences.}


The family of eigenstates $u_k := \mathbf{u}_{k,x_0}$ has $L^2$ norm $\|\chi\|_{L^2(C(x_0,E))}+o(1)$ where $C(x_0,E)$ is equipped with the Lebesgue measure of mass $2\pi$ (Lemma \ref{lemma:l2}) and is the one that saturates the $L^\infty$ norm in Theorem \ref{theorem:magneticzonalstates}. Note that $(u_k)_{k \geq 0}$ is microlocally concentrated on an \emph{immersed $2$-torus}
\begin{equation}
\label{equation:torus}
\begin{split}
\T^2(x_0,E) & := \{\Phi_t(x_0,R_\theta\xi_0) ~:~ \theta \in [0,2\pi], t \in [0,T_E]\} \\
&\simeq \R/(T_E\Z) \times \R/(2\pi\Z).
\end{split}
\end{equation}
It turns out that $\T^2(x_0,E)$ is a \textit{Lagrangian immersion} in $T^*\Sigma$ for the twisted symplectic form, with a maximally degenerate caustic over $x_0$ (see Lemma \ref{lemma:lagrangian-torus}). {\color{black}We thus emphasize that an alternative construction of magnetic zonal states would be to first build a WKB quasimode on the Lagrangian torus $\T^2(x_0,E)$ and then apply a suitable spectral projector $\Pi_E$. The fact that the $L^\infty$ bound is saturated at $x_0$ could then be interpreted as a consequence of the torus $\T^2(x_0,E)\to\Sigma$ having a maximally degenerate caustic at $x_0$.}

\subsubsection{Defect measure} {\color{black}There exists a connection between the saturation of Hörmander's bound and defect measures putting mass on certain Lagrangian tori, as studied in \cite{AIF_2019__69_4_1757_0}.} It is thus natural to describe the semiclassical defect measure $\mu$ associated with the sequence $(u_k)_{k \geq 0}$ of magnetic zonal states \eqref{equation:magnetic-zonal-state}. We introduce $R_E \geq 0$ such that
\begin{equation}
\label{equation:radius}
\cosh R_E := \dfrac{B^2+2E}{B^2-2E}.
\end{equation}
Let $q:\mathbb H^2\to\Sigma$ be the covering map, with $q(i)=x_0$, and set
\[
\D(x_0,R_E):=q\big(\mathbb D(i,R_E)\big), \qquad \gamma_E:=q\big(\partial\mathbb D(i,R_E)\big),
\]
where $\D(p,r)$ is the closed geodesic disk of radius $r \geq 0$ centered at $p$. Then $\D(x_0,R_E)$ is the closed metric disk of radius $R_E$ centered at $x_0$, whereas $\gamma_E$ is the immersed fold caustic
of the projection $\T^2(x_0,E)\to\Sigma$. Note that $\gamma_E$ does not necessarily coincide with $\partial \D(x_0,R_E)$.

\begin{theorem}[Semiclassical defect measure of magnetic zonal states]
\label{theorem:defect}
For any energy $0 < E < E_c$, the following holds:
\begin{enumerate}[label=\emph{(\roman*)}]
    \item The (unnormalized) defect measure $\mu_\chi$ associated with the sequence of magnetic zonal states $(u_k)_{k \geq 0}$ defined in \eqref{equation:magnetic-zonal-state} is supported on $\T^2(x_0,E)$ and equal to
    \[
    \mu_\chi=T_E^{-1} \chi(\theta)^2 \dd t \dd \theta;
    \]
    \item There is an $L^2$-normalized sequence $(u_{k(j)})_{j \geq 0}$ of magnetic zonal states of the form \eqref{equation:magnetic-zonal-state} {\color{black} with $\chi \equiv 1$} such that $\liminf_{j \to +\infty} k(j)^{-1/2}\|u_{k(j)}\|_{L^\infty(\Sigma,L^{\otimes k(j)})} > 0$ and $(u_{k(j)})_{j \geq 0}$ converges in the semiclassical sense to $\mu := (2\pi T_E)^{-1} \dd t \dd \theta$;
    \item The pushforward $\nu := \pi_* \mu$ is a probability measure on $\Sigma$ such that $\nu = \alpha \cdot \vol_{\Sigma}$ where $\alpha \in L^1(\Sigma,\vol)$ is upper-semicontinuous on $\Sigma \setminus \left(\{x_0\} \cup \gamma_E\right)$, $\supp(\alpha)=\D(x_0,R_E)$, and $\alpha$ is singular at $\{x_0\} \cup \gamma_E$. The singularity at $x_0$ is given by
\[
\alpha(x) \sim_{x \to x_0} \dfrac{\sqrt{2/E}}{2\pi T_E d(x,x_0)}.
\]
\end{enumerate}
\end{theorem}

{\color{black}The existence of a sequence of eigenmodes with semiclassical measure given by $\mu_\chi$ or $\mu$ is a straightforward consequence of \cite[Theorem 1.1, item (i)]{Charles-Lefeuvre-25}. Consequently, the point of Theorem \ref{theorem:defect}, item (i) is rather to compute the semiclassical measure associated with the specific sequence constructed in \eqref{equation:magnetic-zonal-state}; the point of item (ii) is to show that one can saturate Hörmander's bound and obtain the Lebesgue measure on $\T^2(x_0,E)$ as the semiclassical defect measure of the sequence at the same time.} Note that $\nu$ has full support on $\Sigma$ if $0 < E < E_c$ is large enough by \eqref{equation:radius} as $R_E \to_{E \to E_c} +\infty$. The function $\alpha$ is explicit, see \eqref{equation:alpha}. The singularity of $\alpha$ near $\gamma_E$ can be completely described as well, and the function $\alpha$ blows up at a rate $\asymp 1/\sqrt{d(x,\gamma_E)}$. However, we postpone this to Proposition \ref{proposition:pushforward} below to avoid burdening the introduction.

Magnetic zonal states are closely analogous to zonal harmonics on the sphere, which saturate Hörmander’s $L^\infty$ bound and are microlocally supported on a Lagrangian torus in phase space, see \cite[Section 4]{Galkowski-Toth-18} for instance. See also \cite{Arnaiz-Macia-22} for a similar phenomenon in the context of harmonic oscillators. However, it is important to note that for classical zonal harmonics the $L^\infty$ norm blows up at two antipodal points on $\mathbb S^2$, a phenomenon caused by the presence of conjugate points. By contrast, for magnetic zonal harmonics the situation is different: the concentration occurs at a single point $x_0$ only (see Lemma \ref{lemma:not-saturated}). This can be naturally explained by the absence of other points on the surface that are ``magnetically'' conjugated to $x_0$.

Finally, let us point out that, as the surface is hyperbolic, there should exist an algebraic method to construct magnetic zonal states using the representation theory of $\mathrm{PSL}(2,\R)$. This is left for future investigation. In particular, this may allow to compute $\Lambda := \limsup_{k \to +\infty} k^{-1/2}\|u_k\|_{L^\infty}$ for sequences of eigenfunctions satisfying \eqref{equation:eigenmodes} at energy $0 \leq E < E_c$. On the sphere, it is known that $\Lambda=(2\pi)^{-1/2}$.

\subsection{Critical energy regime: improved $L^\infty$ bounds}

In the critical energy regime, the \textit{ergodicity} of the magnetic flow on the energy shell $\{E = E_c\}$ leads to a polynomial improvement on the $L^\infty$ norm of eigenfunctions.

\begin{theorem}[$L^\infty$ bounds of magnetic eigenfunctions in the critical regime]
\label{theorem:criticalregime}
    Under the hypothesis \eqref{equation:eigenmodes2}, there exists a constant $C > 0$ such that:
    \begin{equation}
\label{equation:polynomial-improvement}
\|u_k\|_{L^\infty(\Sigma,L^{\otimes k})} \leq C k^{1/2 - \theta \min(\ell,1/15)/155800}.
\end{equation}
\end{theorem}

The polynomial improvement \eqref{equation:polynomial-improvement} for the $L^\infty$ norm of eigenfunctions in the critical energy regime is clearly not optimal and we did not attempt to optimize the proof. It is not clear to us what the optimal polynomial gain should be. To the best of our knowledge, this is also the first instance in the literature where one can observe such a drastic change in the behaviour of $L^\infty$ norms of eigenfunctions of elliptic operators depending on the energy level under consideration. In particular, this raises the question of understanding the transition regime for eigenfunctions as the energy tends (from below) to the critical energy.

Refinements of Hörmander's bound have been investigated in several contexts with an emphasis on negatively-curved manifolds, see \cite{Berard-77, Bonthonneau-17, Canzani-Galkowski-23} among other references. However, polynomial improvements, although conjectured in some settings such as negatively-curved surfaces \cite{Sarnak-95}, or completely integrable surfaces \cite{bourgain1993eigenfunction}, are difficult to derive and only known in specific situations, see \cite{Iwaniec-Sarnak-95,Ingremeau-Vogel-24} for instance. In nonnegative curvature, polynomial improvements have been obtained for tori {\color{black}(for which $L^\infty$ bounds boil down to arithmetic estimates, see \cite{germain2023l2} and references therein)} and in some completely integrable cases, see \cite{chabert2025bounds,chabert2025infty}. {\color{black}Also, for completely integrable systems, norms of joint eigenfunctions have been studied notably in \cite{tacy2019p,galkowski2020pointwise}. See also \cite[Theorem 5]{canzani2021eigenfunction} for logarithmic improvements that do not require joint eigenfunctions.}

In the high energy regime $E > E_c$, one should be able to prove a $\sqrt{\log}$ improvement (that is $\|u_k\|_{L^\infty} \leq C (\log k)^{-{\color{black}1/2}} k^{1/2}\|u_k\|_{L^2}$) as in negative curvature by following the same strategy, see \cite{Berard-77,Bonthonneau-17}. This is not investigated in the present article.

The strategy of proof of Theorem \ref{theorem:criticalregime} is first to bound the $L^\infty$ norm of eigenfunctions by their $L^2$ norms on small balls. Although this result is known in the litterature (see Sogge's paper \cite{sogge2016localized},  the results of which we recall below), we give an alternate proof in Section \ref{section:localized-norm}, which yields a weaker bound than Sogge's result, but in a more general setting. Then, we use \eqref{equation:concentration} to bound the $L^2$ norms on small balls.

\color{black}

\subsection{$L^p$ norm of eigenstates} \label{ssection:lp} With regards to the $L^p$ norm of eigenfunctions satisfying \eqref{equation:eigenmodes}, for $p \in (2,+\infty]$, we recall that the bound of Sogge \cite{sogge1988concerning} asserts that, for general elliptic pseudodifferential operators, one has $\|u_k\|_{L^p} \leq C k^{\gamma(p)} \|u_k\|_{L^2}$ where the exponent $\gamma(p)$ is given by:
{\color{black}\[
\gamma(p) =
\begin{cases}
\dfrac{n-1}{2}\left(\dfrac{1}{2}-\dfrac{{\color{black}1}}{p}\right),
& \begin{aligned}
2 &\le p \leq p_{\mathrm{ST}},
\end{aligned} \\[0.4em]
\dfrac{n-1}{2} - \dfrac{n}{p}
& \begin{aligned}
p_{\mathrm{ST}} &\le p
                \le +\infty,
\end{aligned}
\end{cases}
\]}
and the two behaviors are separated by the Stein-Thomas exponent $p_{\mathrm{ST}} := 2\cdot \frac{n+1}{n-1}$. Here, $n = \dim(\Sigma) = 2$, so $p_{\mathrm{ST}} = 6$. In our context, we obtain:

\begin{theorem}[$L^p$ bounds of magnetic eigenfunctions] The following holds:
\begin{enumerate}[label=\emph{(\roman*)}]
    \item \emph{\textbf{Low energy regime.}} Let $0 \leq E < E_c$. Then, for any $p\in (2, +\infty]$, there exists a sequence $(u_k)_{k\geq 0}$ satisfying \eqref{equation:eigenmodes} such that
    \[
    \liminf_{k\to +\infty} k^{-\gamma(p)} \|u_k\|_{L^p(\Sigma, L^{\otimes k})} > 0.
    \]
    \item \emph{\textbf{Critical energy regime.}} Under the hypothesis \eqref{equation:eigenmodes2}, if $E = E_c$, for any $p\in (2, +\infty]$, {\color{black}there exists a constant $C > 0$ and $\eps = \eps(\theta,\ell,p) > 0$} such that, for any $(u_k)_{k \geq 0}$ satisfying \eqref{equation:eigenmodes2}, 
    \begin{equation}\label{eqpcritic}
        \|u_k\|_{L^p(\Sigma, L^{\otimes k})} \leq C k^{\gamma(p) - \eps}.
    \end{equation}
\end{enumerate}
\end{theorem}

Let us briefly explain how this theorem follows from Theorem \ref{theorem:magneticzonalstates} and existing results in literature. In the low energy regime, for $2 \leq p\leq p_{\mathrm{ST}}$, the existence of a sequence of quasimodes which saturate Sogge's bound is well-known for elliptic pseudodifferential operators in the presence of a closed stable bicharacteristic of the principal symbol, since one can construct Gaussian beams along the closed bicharacteristic, see \cite{germain2023l2}. In our context, the construction adapts since, in the low energy regime, all bicharacteristics are periodic with the same period $2\pi$, hence are closed and stable. Moreover, the Gaussian beams are exact eigenfunctions, so the construction actually yields a sequence of eigenfunctions saturating Sogge's bound. In addition, for $p_{\mathrm{ST}} \leq p \leq +\infty$, it is a general fact that any sequence of eigenmodes which saturate Hörmander's $L^{\infty}$ bound actually saturates Sogge's $L^p$ bound. Indeed, since eigenfunctions are localized in phase space, they satisfy, for any $p \geq 2$, the Bernstein-type inequality
\[\|u_h\|_{L^{\infty}} \lesssim h^{-\frac{2}{p}} \|u_h\|_{L^p}.\]
A proof of this fact in the Euclidean setting is given in \cite[Lemma 2.1]{Bahouri2011} for the flat Laplacian. This extends to any elliptic selfadjoint pseudodifferential operator, after localization in a chart, and bounding the remainder terms with commutator estimates, using for example \cite[Lemma 2.4]{wang2019strichartz}.

{\color{black}That a quantum ergodic behavior leads to improved $L^p$ bounds for small values of $p$ was already observed in the literature, see \cite{sogge2014chapter,blair2017refined,hezari2016lp} for example. In the critical energy regime, to avoid lengthening the present article, we only briefly outline the proof for small $p$, as it follows from an extension of the argument in \cite{sogge2016localized} together with arguments developed in the present paper. First, one needs only prove polynomial improvements for $p = \infty$ and $p = p_{\mathrm{ST}} = 6$, since other cases then follow by interpolation. In the case of an elliptic pseudodifferential operator, \cite[Theorem 1.1, Equation (3.3)]{sogge2016localized} yields the following bound of the $L^p$ norm of eigenmodes, $p = 6, \infty$, by their localized $L^2$ norm on small balls. In the next statement, $(M,g)$ is a compact boundaryless Riemannian surface, $u_\lambda$ is an eigenfunction of $\sqrt{-\Delta_g}$ with eigenvalue $\lambda$, $\mathrm{inj}(M)$ denotes the injectivity radius of $(M,g)$, and $B_r(x)$ is the geodesic ball centered at a point $x\in \Sigma$ of radius $r$:

\begin{align*}
    \|u_\lambda\|_{L^{p_{ST}}(M)} & \leq C \lambda^{\gamma(6)} \left(r^{-\frac{3}{4}} \sup_{x \in M} \|u_\lambda\|_{L^2(B_r(x))}\right)^{2/3}, &  \lambda^{-1} \leq r \leq \mathrm{inj}(M) \\
    \|u_\lambda\|_{L^\infty(M)} & \leq C \lambda^{\gamma(\infty)} \sup_{x\in M} r^{-\frac{1}{2}} \|u_\lambda\|_{L^2(B_r(x))}, & \lambda^{-1} \leq r \leq \mathrm{inj}(M).
\end{align*}

Modulo small error terms of the size $\mc{O}(h^\eta)$, $\eta > 0$, and perhaps with a restriction on the allowed size of the geodesic balls, it is likely that one can generalize this result in our context, i.e. for twisted semiclassical pseudodifferential operators with principal symbol $p$ satisfying $\partial_\xi p(x,\xi) \neq 0$. We may then use the improved localized $L^2$ estimates proved in \S\ref{improvsmallball} to finally obtain a polynomial improvement for $p = 6, \infty$.

As the proof of Sogge is really written out in the context of standard elliptic selfadjoint pseudodifferential operators, we believe that one cannot directly apply it to our setting, although it probably adapts. Hence, we have chosen to give an alternative proof of a weaker version of Sogge's result for $p = \infty$ in \S\ref{section:localized-norm}, and we do not write out the proof of $p = 6$ in order not to lengthen the present article.}

\color{black}

\subsection{Organization of the paper} In \S\ref{section:preliminaries}, we recall standard facts on the semiclassical calculus of pseudodifferential operators on line bundles $L^{\otimes k} \to \Sigma$ (with semiclassical parameter $h:=1/k$). We also recall Weinstein's averaging method for the magnetic Laplacian, together with some geometric and dynamical properties of the magnetic flow on $T^*\Sigma$. In \S\ref{section:localized-norm}, we prove a general estimate for the $L^\infty$ norm of eigenfunctions of semiclassical pseudodifferential operators in terms of a localized $L^2$ norm which may be of independent interest. Finally, Theorems \ref{theorem:magneticzonalstates}, \ref{theorem:defect} and \ref{theorem:criticalregime} are proved in \S\ref{section:proof}.

\medskip

\noindent \textbf{Acknowledgement:} We thank Yves Colin de Verdière for suggesting to study $L^\infty$ norms of magnetic eigenfunctions, and Laurent Charles for fruitful discussions. We are also grateful to Maxime Ingremeau and Gabriel Rivière for comments on an earlier version of this draft. {\color{black}We are grateful to the anonymous referees for their valuable comments which improved the presentation of this manuscript.} This project was supported by the European Research Council (ERC) under the European Union’s Horizon 2020 research and innovation programme (Grant agreement no. 101162990 — ADG). 

\medskip

\noindent \textbf{Compliance with Ethical Standards:} The authors report that there is no conflict of interest to declare. The authors have no financial or proprietary interests in any material discussed in this article.

\section{Preliminaries} \label{section:preliminaries}

\subsection{Twisted pseudodifferential operators} \label{ssection:twisted-pseudos}

We briefly recall the main properties of the algebra of twisted semiclassical operators on $\Sigma$, first introduced by Charles in \cite{Charles-00}. For a detailed study, we refer to \cite[Chapter 3, Section 3.2]{Cekic-Lefeuvre-24}. Throughout this paragraph, $\Sigma$ denotes a closed manifold, $L \to \Sigma$ is a Hermitian complex line bundle equipped with a unitary connection $\nabla$. {\color{black}For $m \in \R$, let $S^m(T^*\Sigma)$ denote the class of symbols of order $m$ (see \cite[Section E.1.2]{Dyatlov-Zworski-19}), and $\Psi^m_h(\Sigma)$ be the corresponding set of standard semiclassical pseudodifferential operators with semiclassical parameter $h > 0$ (see \cite[Section E.1.7]{Dyatlov-Zworski-19}).}

\begin{definition}[Twisted semiclassical $\Psi$DOs] A family of operators
\[
\mathbf{P}_k : C^\infty(\Sigma, L^{\otimes k}) \to C^\infty(\Sigma, L^{\otimes k})
\]
is a family of \emph{twisted} semiclassical pseudodifferential operators of order $m \in \R$ if the following holds: for any small contractible open set $U \subset \Sigma$, any trivializing section $s \in C^\infty(U,L)$ such that $|s|=1$, and any cutoff function $\chi \in C^\infty_{\comp}(U)$ there exists $P_k \in \Psi^m_{k^{-1}}(\Sigma)$ a family of standard semiclassical pseudodifferential operators with semiclassical parameter $h=1/k$ such that for all $f \in C^\infty_{\comp}(U)$:
\[
\chi \mathbf{P}_k(f s^{\otimes k}) = P_k(f) s^{\otimes k}.
\]
\end{definition}

The set of all such operators form an algebra, that is it is stable under multiplication and addition. The principal symbol of $\mathbf{P}_k$ is well-defined and independent of the choice of trivializing section $s$ if we set
\begin{equation}
\label{equation:psymbol}
\sigma_{\mathbf{P}_k}(x,\xi) := \sigma_{P_k}(x,\xi+\beta(x)),
\end{equation}
where $\beta \in C^\infty(U,T^*U)$ is the connection $1$-form such that $\nabla s = -i\beta \otimes s$. In addition, all standard results of semiclassical analysis still hold in this context (invertibility of elliptic operators, Calderon-Vaillancourt, Egorov, etc.), see \cite[Proposition 3.2.3]{Cekic-Lefeuvre-24}.

Conversely, given a symbol $p \in S^m(T^*\Sigma)$, one can define a magnetic quantization as follows. Let $g$ be an arbitrary background metric on $\Sigma$, $\chi \in C^\infty(\Sigma \times \Sigma)$ a smooth function equal to $1$ near the diagonal $\Delta \subset \Sigma \times \Sigma$ and supported in $\{d(x,y) < \iota(g)/2026\}$, where $\iota(g)$ is the injectivity radius of the metric $g$. For $k \geq 0$, and $x,y \in \mathrm{supp}(\chi)$, let $\tau^{\color{black}(k)}_{x \to y} : L^k_x \to L^k_y$ denote the parallel transport with respect to $\nabla^{\otimes k}$ along the unique $g$-geodesic joining $x$ to $y$. We can then quantize
\[
\mathbf{P}_k := \Op_k(p) : C^\infty(\Sigma,L^{\otimes k}) \to C^\infty(\Sigma,L^{\otimes k})
\]
by setting for $f \in C^\infty(\Sigma, L^{\otimes k})$:
\[
\Op_{k}({\color{black}p}) f(x) := \dfrac{1}{{\color{black}(2\pi k^{-1})^2}} \int_{\Sigma} \int_{T^*_x\Sigma} e^{-ik\xi(\exp^{-1}_x(y))} {\color{black}p}(x,\xi) \tau^{\color{black}(k)}_{y \to x}(f(y)) \chi(x,y) \dd\xi \dd\vol_{\Sigma}(y),
\]
where $\dd \vol_{\Sigma}(y)$ is the Riemannian volume, and $\dd\xi$ the induced volume in the fibers of $T^*\Sigma$. It is a standard calculation to verify that
\[
\sigma_{\Op_{k}(p)} = [p] \in S^m(T^*\Sigma)/k^{-1}S^{m-1}(T^*\Sigma).
\]

Finally, given a sequence $(u_k)_{k \geq 0}$ of states $u_k \in C^\infty(\Sigma,L^{\otimes k})$ such that $\|u_k\|_{L^2}=1$, we say that the measure $\mu$ on $T^*\Sigma$ is a semiclassical defect measure associated with the sequence if, up to extraction, the following holds: for all $p \in C^\infty_{\comp}(T^*\Sigma)$,
\[
\langle\Op_k(p)u_k,u_k\rangle_{L^2(\Sigma,L^{\otimes k})} \to_{k \to +\infty} \int_{T^*\Sigma} p(x,\xi) \dd\mu(x,\xi).
\]
We denote this convergence by $u_k \rightharpoonup_{k \to +\infty} \mu$.

\subsection{The averaging method} \label{ssection:averaging} The following paragraph is inspired by Weinstein's averaging method \cite{Weinstein-77} and summarizes \cite[Section 4.1]{Charles-Lefeuvre-25}.

Let $0 \leq E < E_c$. For $0 \leq m \leq N_k = \lfloor kB \rfloor$, we let $\Pi_{k,m}$ be the orthogonal projector (in $L^2(\Sigma,L^{\otimes k})$) onto the eigenspace of $\Delta_k$ associated with the eigenvalue $\lambda_{k,m}$ (see \eqref{equation:lambdamk}). Let 
\[
\mathbf{A}_k := k^{-1} \sum_{m=0}^{N_k-1} m \Pi_{k,m}
\] 

Observe that, by construction, the spectrum of $k \mathbf{A}_k$ is contained in $\Z_{\geq 0}$. In particular, this implies that $e^{2i\pi k \mathbf{A}_k} = \mathbf{1}$. {\color{black}In addition, it satisfies
\[
\Pi_{k,m} = \dfrac{1}{2\pi} \int_0^{2\pi} e^{-imt} e^{itk\mathbf{A}_k} \dd t.
\]
The connection with the magnetic Laplacian is that, using \eqref{equation:lambdamk}, it satisfies the algebraic identity
\begin{equation}
    \label{equation:algebraic-identity}
    k^{-2}\Delta_k = B(\mathbf{A}_k + \tfrac{1}{2}k^{-1})-\tfrac{1}{2}\mathbf{A}_k(\mathbf{A}_k+k^{-1}),
\end{equation}
where the identity holds on
\[
\bigoplus_{m=0}^{\lfloor k B\rfloor - 1} \ker(\Delta_k - \lambda_{k,m}),
\]
see \cite[Equation (4.4)]{Charles-Lefeuvre-25} for details.

We emphasize that, formally speaking, $\textbf{A}_k$ is \emph{not} a twisted semiclassical pseudodifferential operator, due to the sharp spectral cutoff involved in its definition. However, letting $D^*\Sigma := \{(x,\xi) \in T^*\Sigma ~:~ |\xi| < B\}$, and using \eqref{equation:algebraic-identity}, it is natural to introduce the ``principal symbol'' of $\textbf{A}_k$ as the function $a$ on $D^*\Sigma$ defined through the relation $p(x,\xi) = \frac{1}{2}|\xi|^2 = \beta(a(x,\xi))$, where $\beta : s \in [0,B]\mapsto Bs - \frac{1}{2} s^2 \in [0,\frac{1}{2} B^2]$, see \cite[Section 4.1.1]{Charles-Lefeuvre-25}. It can then be shown that for any $\chi \in C_{\mathrm{comp}}^\infty(D^*\Sigma)$, $\Op_k(\chi) \textbf{A}_k$ \emph{is} a twisted compactly supported semiclassical pseudodifferential operator with semiclassical principal symbol $\chi a$, see \cite[Lemma 4.1]{Charles-Lefeuvre-25}.

\begin{remark}
\label{remark:cutoff}
In the following, with a slight abuse of notation, we omit the cutoff $\chi$ and regard $\textbf{A}_k$ as a genuine twisted semiclassical operator. This will not affect the proofs as we will be working on energy shells $E < E_c$, that is away from the boundary of $D^*\Sigma$. In particular, this fact will be used in the following form: for any function $(u_k)_{k \geq 0}$ with semiclassical wavefront set $\WF_k(u_k)$ contained in a \emph{compact} subset $K \subset D^*\Sigma$, one has $\textbf{A}_k u_k = \Op_k(\chi) \textbf{A}_k u_k + \mc{O}_{C^\infty}(k^{-\infty})$, where $\chi \in C^\infty_{\comp}(D^*\Sigma)$ is equal to $1$ on an open neighborhood of $K$.
\end{remark}}

The Hamiltonian flow $(\Phi_t^a)_{t\in \R}$ generated by $a$ (with respect to the twisted symplectic form $\Omega$ introduced in \eqref{equation:twisted-form}) on $D^*\Sigma$ is $2\pi$-periodic. It is a reparametrization of the Hamiltonian flow $(\Phi_t)_{t \in \R}$ of $p$ which is $T_E$-periodic with $T_E = 2\pi(B^2-2E)^{-1/2}$. We will also need a fact taken from \cite[Lemma 4.3]{Charles-Lefeuvre-25}. Given $f \in C^\infty_{\comp}(D^*\Sigma)$, we let
\[
\langle f \rangle(z) := \dfrac{1}{2\pi} \int_0^{2\pi} f(\Phi_t^a(z)) \dd t.
\]
Fix $\eps > 0$. We then have, for all $0 \leq m \leq (1-\eps)kB$, and $b \in C^\infty_{\comp}(D^*\Sigma)$:
\begin{equation}
    \label{equation:integral}
    \Pi_{k,m} \Op_k(b) \Pi_{k,m} = \Op_k(\langle b \rangle) \Pi_{k,m} + \mc{O}_{L^2}(k^{-1}).
\end{equation}
{\color{black}
We emphasize that $\eps > 0$ avoids the problem of the spectral cutoff raised in Remark \ref{remark:cutoff}.}


\subsection{Lagrangian tori}

\label{ssection:lagrangian-tori}

Throughout this section, we still assume that ${\color{black}0 <} E < E_c$, that is $B^2-2E > 0$. Recall from \eqref{equation:torus} that
\[
\T^2(x_0,E) = \{\Phi_t(x_0,R_\theta \xi_0) ~:~ \theta \in [0,2\pi), t \in [0,T_E]\}
\]
is an immersed $2$-torus (where $T_E := 2\pi(B^2-2E)^{-1/2}$ is the period of the magnetic flow on the energy shell $\{p=E\}$). {\color{black}We refer to Figure \ref{figure:2} for an illustration of $\T^2(x_0,E)$ on the universal cover of the surface.}

\begin{lemma}
\label{lemma:lagrangian-torus}
The immersed $2$-torus $\T^2(x_0,E)$ is Lagrangian for the twisted symplectic form $\Omega$.
\end{lemma}

\begin{proof}
At $z \in \T^2(x_0,E)$, the Hamiltonian vector field $H_p^\Omega(z)$ is tangent to $\T^2$ by construction; we can also find another vector $Y(z)$, tangent to $\T^2(x_0,E)$ (hence to $\{p=E\}$) and transverse to $H_p^\Omega(z)$. We thus find $\Omega(Y,H_p^\Omega) = dp(Y) = 0$ as $Y$ is tangent to a level set of $p$.
\end{proof}

Let $\pi : T^*\Sigma \to \Sigma$ be the footpoint projection. We introduce $R_E \geq 0$ such that:
\[
\cosh R_E = \dfrac{B^2+2E}{B^2-2E}.
\]
Note that $R_0=0$ and $R_E \to_{E \to E_c} +\infty$. In the following lemma, $\D(x,R) \subset \Sigma$ is the geodesic disk of radius $R$, centered at $x \in \Sigma$.

\begin{lemma}
\label{lemma:footprojection}
Let $\pi : \T^2(x_0,E) \to \Sigma$ be the footpoint projection. Then $\pi(\T^2(x_0,E)) = \D(x_0,R_E)$.
\end{lemma}

\begin{proof}
Let $\lambda := \sqrt{2E}$. Modulo conjugation by a fiberwise rescaling, the magnetic flow $(\Phi_t)_{t \in \R}$ on the energy shell
\[
S_\lambda^*\Sigma := \{|\xi|=\lambda\}=\{p=E\}
\]
is conjugate to the flow generated by ${\color{black}\lambda}X-{\color{black}B}V$ on $S^*\Sigma := S_1^*\Sigma$, where $X$ is the geodesic vector field, and $V$ is the generator of the $2\pi$-periodic rotation in $T^*\Sigma$, see \cite[Proposition 2.5]{Charles-Lefeuvre-25}. As $\Sigma$ is hyperbolic, we may write $S^*\Sigma \simeq S\Sigma = \Gamma\backslash\mathrm{PSL}(2,\R)$ for some Fuchsian group $\Gamma < \mathrm{PSL}(2,\R)$. The vector field ${\color{black}\lambda}X-{\color{black}B}V$ can then be identified with the element of $\mathfrak{sl}(2,\R)$ given by the matrix
\[
F_\lambda := {\color{black}\lambda}X-{\color{black}B}V = \begin{pmatrix} \lambda/2 & -B/2 \\ B/2 & -\lambda/2 \end{pmatrix},
\]
see the proof of \cite[Proposition 2.5]{Charles-Lefeuvre-25}.

We now argue on the universal cover of $\Sigma$, identified with the Poincaré upper half-plane $\HH^2$, and lift $x_0$ to $i \in \HH$. Consider the (unit) tangent vector $v_0 \in T_i\HH^2$ given in complex coordinates by $v_0=i$. Recall that there is an identification $\mathrm{PSL}(2,\R) \simeq S\HH^2$ given by $\gamma \mapsto (\gamma.i, d\gamma_i v_0)$, where the $\mathrm{PSL}(2,\R)$ action has the following explicit expression:
\[
\gamma.(z,v) = (\gamma.z, d\gamma_z v) = \left(\dfrac{az+b}{cz+d}, \dfrac{v}{(cz+d)^2}\right),
\]
see \cite[Chapter 1]{Katok-92}.

We thus find that the trajectory of $(i,v_0)$ under the flow generated by $F_\lambda$ is
\begin{equation}
\label{equation:trajectory}
e^{t F_\lambda}(i,i) =: (z(t),v(t)).
\end{equation}
The eigenvalues of $F_\lambda$ are given by $\pm \tfrac{i}{2}\gamma$, $\gamma := (B^2-\lambda^2)^{1/2}$. As a consequence, $e^{2\pi/\gamma \cdot F_\lambda} = -\mathbf{1} = \mathbf{1}$ in $\mathrm{PSL}(2,\R)$, that is the trajectory of $(i,v_0)$ has period $2\pi/\gamma = T_E$. By homogeneity of $\HH^2$, the function
\begin{equation}
    \label{equation:phit-h}
\phi(t) := d_{\HH^2}(i,z(t))
\end{equation}
is seen to be maximized at time $t = \pi/\gamma=T_E/2$. A quick computation shows that
\[
e^{\pi/\gamma \cdot F_\lambda} = \gamma^{-1} \begin{pmatrix} \lambda & - B \\ B & -\lambda \end{pmatrix}.
\]
As a consequence, 
\[
z(\pi/\gamma) = \dfrac{\lambda i-B}{Bi - \lambda} = \dfrac{2\lambda B+i(B^2-\lambda^2)}{\lambda^2+B^2}.
\]

The hyperbolic distance $d_{\HH^2}(i,z(\pi/\gamma))$ between $i$ and $z(\pi/\gamma)$ satisfies (see \cite[Theorem 1.2.6]{Katok-92}):
\[
\cosh d_{\HH^2}(i,z(\pi/\gamma)) = 1 + \dfrac{|z-i|^2}{2 \Im(z) \Im(i)} = \dfrac{B^2+\lambda^2}{B^2-\lambda^2}.
\]
Rotating the vector $v_0 \in T_i \HH^2$, that is considering the trajectories of $(i,R_\theta v_0)$, where $R_\theta$ denotes the rotation by angle $\theta \in [0,2\pi]$, one sees that the union of the trajectories (under the flow generated by $F_\lambda$) projects onto the geodesic disk of radius $\cosh^{-1}\left(\tfrac{B^2+\lambda^2}{B^2-\lambda^2}\right)$, centered at $i \in \HH^2$. The conclusion is then immediate as we consider the energy shell $\{\tfrac{1}{2}|\xi|^2=E\}$, that is $\lambda^2 = 2E$.  
\end{proof}

For later use, we also record a useful computation on the derivatives of the function $\phi(t) = d_{\HH^2}(i,z(t))$ introduced in \eqref{equation:phit-h} at $t=\pi/\gamma$, where $z(t)$ is defined in \eqref{equation:trajectory}.

\begin{lemma}
\label{lemma:calculs}
We have:
\[
\phi'(\pi/\gamma) = 0, \qquad \phi''(\pi/\gamma) = \dfrac{\sqrt{2E}}{2B}\left(2E-B^2\right) < 0.
\]
\end{lemma}

\begin{proof}
Consider $r : \HH^2 \setminus \{i\} \to (0,\infty)$ defined by $r(z) := d_{\HH^2}(i,z)$. The vector $\nabla r(z)$ is the tangent vector at $z$ to the unique hyperbolic geodesic joining $i$ to $z$. Note that $\phi(t) = r(z(t))$. We thus find:
\[
\phi'(\pi/\gamma) := \langle \nabla r(z(\pi/\gamma)), \dot{z}(\pi/\gamma)\rangle = 0
\]
as $\dot{z}(\pi/\gamma)$ is tangent to $\partial \D^2(i,R_E)$, which is orthogonal to the geodesic joining $i$ to $z(\pi/\gamma)$.

In addition, we have
\[
\phi''(t) = \langle \nabla^2 r(z(t)) \dot{z}(t), \dot{z}(t)\rangle + \langle \nabla r(z(t)), \nabla_{\dot{z}(t)} \dot{z}(t)\rangle,
\]
where $\nabla^2 r$ is the Hessian of $r$. Note that, when restricted to tangent vectors to the circle $\{z \in \HH^2 ~:~ d_{\HH^2}(z,i)=R_E\}$ in $\HH^2$, it coincides with the shape operator of this circle (the second fundamental form). {\color{black}Using polar coordinates around $i$, the hyperbolic metric writes $g_{\HH^2} = dr^2 + \sinh(r)^2 d\theta^2$, from which one deduces that $\nabla^2 r = \coth(R_E)g_{\HH^2}$ when applied to tangent vectors to the circle (see \cite[Lemma 3.1]{alvarez2026foliatedplateauproblemssurface} for instance)}. Since $t \mapsto z(t)$ is a magnetic geodesic, it satisfies by definition the equation $\nabla_{\dot{z}(t)} \dot{z}(t) = - B j_{z(t)} \dot{z}(t)$, where $j \in C^\infty(\HH^2, \mathrm{End}(T\HH^2))$ is the almost complex structure on $\HH^2$ (the fiberwise rotation by $\pi/2$), see \cite[Equation (2.5)]{Charles-Lefeuvre-25}. Using the formula $\coth(\cosh^{-1}(x))= x (x^2-1)^{-1/2}$ for $x \geq 1$, we find:
\[
\phi''(\pi/\gamma) = \coth(R_E)|\dot{z}(t)|^2 - B|\dot{z}(t)| = \dfrac{\sqrt{2E}}{2B}\left(2E-B^2\right),
\]
since $|\dot{z}(t)|=\sqrt{2E}$ on the energy shell $\{p=E\}$.
\end{proof}

\begin{lemma}
\label{lemma:nmber-bounded}
There exists an integer $N_E \geq 0$ such that for any $x \in \Sigma, x \neq x_0$, the number of preimages $\pi^{-1}(x) \cap \T^2(x_0,E)$ is finite and bounded by $N_E$.
\end{lemma}

\begin{proof}
We argue on the universal cover $\HH^2$ of $\Sigma$. Consider $\pi : \T^2(i,E) \to \D(i,R_E)$. A quick modification of the arguments in the proof of Lemma \ref{lemma:footprojection} shows that any point $x \in \D(i,R_E)$, $x \neq i$ and $x \notin \partial \D(i,R_E)$ has exactly two preimages. A point $x \in \partial \D(i,R_E)$ has exactly one preimage. Covering $\D(i,R_E)$ by a finite number $N'_E$ of fundamental domains for $\Sigma$, we then see that, on the surface $\Sigma$ itself, any point $x \neq x_0$ has at most $N_E := 2N'_E$ preimages under the projection $\pi : \T^2(x_0,E) \to \Sigma$.
\end{proof}

Let $q:\mathbb H^2\to\Sigma$ be the covering map, with $q(i)=x_0$, and set
\[
\D(x_0,R_E):=q\big(\mathbb D(i,R_E)\big), \qquad \gamma_E:=q\big(\partial\mathbb D(i,R_E)\big).
\]
That is $\D(x_0,R_E)$ is the closed metric disk of radius $R_E$ centered at $x_0$, and $\gamma_E$ is the immersed fold caustic
of the projection $\T^2(x_0,E)\to\Sigma$. Note that $\gamma_E$ is a (piecewise) smooth immersed curve in $\Sigma$. We say that $y_0 \in \gamma_E$ is regular if $\gamma_E$ is smooth near $y_0$ and singular otherwise. Singular points on $\Sigma$ correspond to those having multiple preimages in the universal cover $\HH^2$ lying on $\partial \D(i,R_E)$.

Note that, locally, near singular points, we may write $\gamma_E = \cup_{i=1}^{J(y_0)} \gamma_i$, where the $\gamma_i$'s are smooth embedded curves. They correspond to local projections of pieces of $\partial \D(i,R_E) \subset \HH^2$. The \emph{order} $J(y_0) \geq 1$ of the singularity is equal to the number of preimages of $y_0$ on $\partial \D(i,R_E) \subset \HH^2$.

\begin{figure}[h!]
\begin{center}
\includegraphics[scale=1]{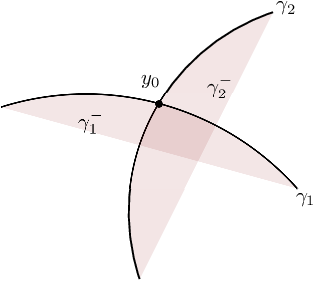}
\caption{A crossing of $\gamma_E$ with a singularity of order $2$.}
\end{center}
\label{figure:crossing}
\end{figure}

Let $\mathbb{A}^- \subset \HH^2$ (resp. $\mathbb{A}^+ \subset \HH^2$) be the annulus consisting of the set of points $z \in \HH^2$ such that $R_E-\eps < d_{\HH^2}(i,z) < R_E$ for some $\eps > 0$ small enough (resp. $R_E < d_{\HH^2}(i,z) < R_E+\eps$). The projections of $\mathbb{A}^\pm$ to the surface $\Sigma$ define two half neighborhoods of $\gamma_E$, which we denote by $\gamma_E^\pm$. Note that they may \emph{not} be disjoint. Near a singular point $y_0 \in \gamma_E$, we may actually decompose $\gamma_E^\pm = \cup_{i=1}^{J(y_0)} \gamma_i^\pm$.

Let $\mathrm{Leb}_{\T^2}$ be the Lebesgue measure $\dd t \otimes \dd \theta$ on the immersed $2$-torus $\T^2(x_0,E)$. We compute in the next proposition its projection to the surface via $\pi : T^*\Sigma \to \Sigma$.

\begin{proposition}[Pushforward of the Lebesgue measure]
\label{proposition:pushforward}
The pushforward $\nu := \pi_* \mathrm{Leb}_{\T^2}$ is equal to $\nu = \alpha \vol_{\Sigma}$, where $\alpha \in L^1(\Sigma,\vol)$ is explicit (see \eqref{equation:alpha}), with support equal to $\D(x_0,R_E)$, upper semicontinuous on $\Sigma \setminus \left(\{x_0\} \cup \gamma_E\right)$. In addition:
\begin{itemize}
    \item Near $x_0$:
    \[
    \alpha(x) \sim_{x \to x_0} \dfrac{\sqrt{2/E}}{d(x_0,x)}.
    \]
    \item Near $y_0 \in \gamma_E$:
    \[
    \alpha(x) = \mc{O}(1), \qquad x \to y_0, x \in \textcolor{black}{\gamma^+_E, x \notin \gamma^-_E}, 
    \]
    and
    \begin{equation}
        \label{equation:adding}
    \alpha(x) \sim \dfrac{1}{E} \left(\dfrac{\sqrt{2E}(B^2-2E)}{4B}\right)^{1/2}\sum_{j=1}^{J(y_0)} \dfrac{\mathbf{1}_{\gamma_j^-}(x)}{d_{\Sigma}(x,\gamma_j)^{1/2}}, \qquad x \to y_0, x \in \gamma_E^-,
    \end{equation}
   where $\mathbf{1}_{Y}$ denotes the indicator function of a set $Y \subset \Sigma$.
\end{itemize} 
\end{proposition}

{\color{black}Note that the above statement is ambiguous if $x_0 \in \gamma_E$. In this case, only the estimate $\alpha(x) \sim_{x \to x_0} \sqrt{2/E}d(x_0,x)^{-1}$ remains valid near $x_0$ if $x \to x_0$ radially.}

\begin{figure}[h!]
\begin{center}
\includegraphics[scale=0.8]{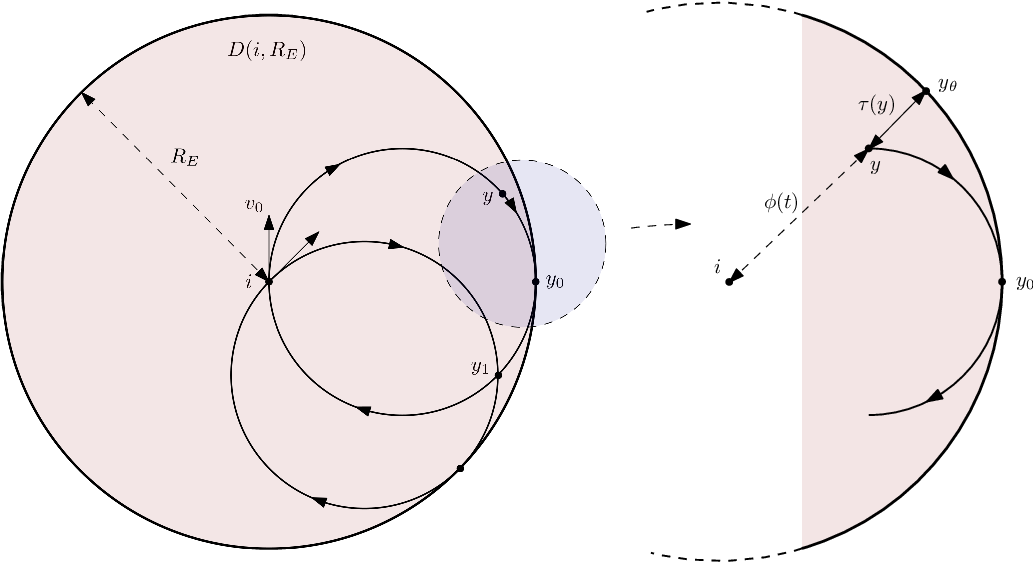}
\caption{Left-hand side: two $T_E$-periodic magnetic trajectories inside the geodesic disk $\D(i,R_E)$ are depicted in the Poincaré hyperbolic disk model. The point $y_1$ in the interior of the disk has two preimages under the map $\Psi$ but the point $y_0 \in \partial \D(i,R_E)$ has only one. Right-hand side: zooming in near the boundary $\partial \D(i,R_E)$. The various parameters introduced in the proof of Proposition \ref{proposition:pushforward} are depicted.}
\label{figure:2}
\end{center}
\end{figure}

\begin{proof}
On the energy shell $\{p=E\}$, we let $X_\perp := [X,V]$, \textcolor{black}{where we recall that $X$ is the generator of the geodesic flow, and $V$ is the generator of the fiberwise rotation}. Then $(X,X_\perp,V)$ is an oriented frame of $\{p=E\}$, see \cite[Lemma 2.4]{Charles-Lefeuvre-25}. In addition, one has $H_p^\Omega = X -BV$, see \cite[Lemma 2.6]{Charles-Lefeuvre-25}. \textcolor{black}{By definition of the geodesic vector field,} $d\pi(X(x,\xi)) = \xi^\sharp$, where $\sharp : T^*\Sigma \to \Sigma$ is the musical isomorphism; similarly, it can be verified that $d\pi(X_\perp(x,\xi)) = R_{\pi/2} \xi^\sharp$.

We now use the parametrization of $\T^2(x_0,E)$ by $\theta \in \R/2\pi\Z$ and $t \in [0,T_E]$, that is $(\theta,t) \mapsto \Phi_t(x_0,R_\theta \xi_0)$ where $\xi_0$ is an arbitrary point on $C(x_0,E)$. Note that the footpoint projection is degenerate at $t=0$ as the whole circle $C(x_0,E) \subset \T^2(x_0,E)$ projects onto $x_0$.

First, we claim that
\begin{equation}
    \label{equation:local-diffeo}
\Psi : [0,2\pi] \times ((0,T_E) \setminus \{T_E/2\}) \ni (\theta,t) \mapsto \pi(\Phi_t(x_0,R_\theta \xi_0)) \in \Sigma
\end{equation}
is a local diffeomorphism, namely that its differential is invertible. Let $(y,\eta) := \Phi_t(x_0,R_\theta \xi_0) \in \T^2(x_0,E)$, $y = \Psi(\theta,t)$. Note that $\partial_t \Psi(\theta,t) = d\pi(H_p^\Omega(y,\eta)) = \eta^\sharp \in T_y\Sigma$. In addition, by \cite[Lemma 2.7]{Charles-Lefeuvre-25} (see the computations \cite[Equations (2.10)-(2.18)]{Charles-Lefeuvre-25} with initial conditions $a(0)=b(0)=d(0)=0, c(0)=1$), we have:
\[
\partial_\theta (\Phi_t(x_0,R_\theta \xi_0)) = a(t)X(y,\eta) + b(t)X_\perp(y,\eta) + c(t)V(y,\eta) \in T_{(y,\eta)}\T^2(x_0,E),
\]
where
\begin{equation}
    \label{equation:b}
b(t) = \dfrac{\sin((B^2-2E)^{1/2}t)}{(B^2-2E)^{1/2}}, \quad a(t) = -B\int_0^t b(s)\dd s, \quad c(t) = 1 + 2E\int_0^t b(s)\dd s.
\end{equation}
Hence
\[
\partial_\theta \Psi(\theta,t) = a(t) \eta^\sharp + b(t) R_{\pi/2} \eta^\sharp,
\]
and thus $b(t) \neq 0$ for $t \in (0,T_E) \setminus \{T_E/2\}$, where $T_E = 2\pi(B^2-2E)^{-1/2}$. Note that $t=0$ or $t=T_E$ corresponds to $y=x_0$ while $t=T_E/2$ corresponds to $y \in \gamma_E$. At $t=T_E/2$, we find that $b(T_E/2)=0$ and
\[
a(T_E/2) = -B \int_0^{T_E/2} b(s)\dd s=-\dfrac{2B}{B^2-2E} \neq 0.
\]
The differential of $\Psi$ is thus degenerate at $t=T_E/2$ and has rank $1$.

Consider now $y \neq x_0$ and $y \notin \gamma_E$. As $|\eta^\sharp| = \sqrt{2E}$, the Riemannian volume form at $y$ is given by $\vol_{\Sigma}(y) = (2E)^{-1} \eta^\sharp \wedge R_{\pi/2} \eta^\sharp$. We thus find that $|\det d\Psi(t,\theta)| = 2E |b(t)|$. Using the expression \eqref{equation:b} yields:
\begin{equation}
    \label{equation:alpha}
\pi_* \mathrm{Leb}_{\T^2} (y) = \dfrac{1}{2E} \sum_{\Psi(\theta_i, t_i)=y} \dfrac{(B^2-2E)^{1/2}}{|\sin((B^2-2E)^{1/2}t_i)|} \vol_{\Sigma}(y) = \alpha(y) \vol_{\Sigma}(y),
\end{equation}
where the sum runs over the $N(y)$-preimages of $y$ under $\Psi$. This number is bounded by $N(y) \leq N_E$ (Lemma \ref{lemma:nmber-bounded}), and $y \mapsto N(y)$ is upper semi-continuous. This proves the claim on $\alpha$ when $y \notin \{x_0\} \cup \gamma_E$.

We now study $\alpha$ for $y$ close to $\gamma_E$. We argue on the universal cover as above. Let $y_0 = \Psi(\theta_0,T_E/2) \in \partial \D(i,R_E)$. We may parametrize the (half-)neighborhood of $y_0$ contained in $\D(i,R_E)$ by first considering $y_\theta := \Psi(\theta,T_E/2) \in \partial \D(i,R_E)$ (with $y_{\theta_0} = y_0$), and then the point at distance $\tau \geq 0$ from $y_\theta$ on the (unique) geodesic of length $R_E$ joining $i$ to $y_\theta$. Notice that for $y$ close to $y_0$, $\tau(y) = d_{\HH^2}(y,\gamma_E)$ corresponds the distance between $y$ and $\gamma_E$. We also write $t(y)$ for the time such that $y = \Psi(\theta,t(y))$.

We let $\tau(\theta,t) = \tau(t)$ be the corresponding parameter for the point $y=\Psi(\theta,t)$. By construction, $\tau(t) = R_E-\phi(t)$ so, using Lemma \ref{lemma:calculs}, we find that $\tau(T_E/2)=0, \tau'(T_E/2)=0$ and $\tau''(T_E/2) = c := \tfrac{\sqrt{2E}}{2B}\left(B^2-2E\right) > 0$. Hence, near $t=T_E/2$, we may write $\tau(t) = \tfrac{1}{2} c (t-T_E/2)^2 + \mc{O}((t-T_E/2)^3)$ , and thus $|t(\tau)- T_E/2| = \sqrt{2 \tau/c} + \mc{O}(\tau)$. Using \eqref{equation:alpha} combined with the fact that any point $y$ near $\gamma_E$ has two preimages, and the relation $\tau(y) = d_{\HH^2}(y,\gamma_E)$, we obtain for $y$ close to $y_0$ and $y \in \D(i,R_E)$:
\[
\begin{split}
\alpha(y) = \dfrac{1}{E} \dfrac{(B^2-2E)^{1/2}}{|\sin((B^2-2E)^{1/2}t(y))|} & \sim_{y \to y_0} \dfrac{1}{E|T_E/2-t(y)|} \\
& \sim_{y \to y_0} \dfrac{1}{E}\left(\dfrac{c}{2 d_{\HH^2}(y,\gamma_E)}\right)^{1/2} \\
& = \dfrac{1}{E} \left(\dfrac{\sqrt{2E}(B^2-2E)}{4Bd_{\HH^2}(y,\gamma_E)}\right)^{1/2}.
\end{split}
\]
Note that $\alpha(y) = 0$ if $y \notin \D(i,R_E)$. Going back to the surface $\Sigma$, given $y_0 \in \gamma_E$, the claim \eqref{equation:adding} follows by adding up all the contributions coming from the preimages if $y_0$ is singular.

Finally, for $y$ close to $x_0$, we use again \eqref{equation:alpha}. There is a first preimage obtained for $t(y) \sim_{y \to x_0} d(y,x_0)/\sqrt{2E}$ (note that the geodesic speed is $\sqrt{2E}$ on the energy shell $\{p=E\}$), and another one at $T_E-t(y)$. The other preimages are obtained for times bounded away from $0$ or $T_E$. In addition, if $x_0\in \gamma_E$, the possible singularities coming from times $t(y)$ close to $T_E/2$ are of lower order and therefore negligible. We thus find that:
\[
\alpha(y) \sim_{y \to x_0} \dfrac{\sqrt{2/E}}{d(x_0,y)}.
\]
This completes the proof.
\end{proof}


\section{Bounding the $L^\infty$ norm by spatially localized $L^2$ norms}

\label{section:localized-norm}

Let $(M,g)$ be a compact Riemannian manifold of dimension $n$. Let $P_h \in \textcolor{black}{\Psi^m_h(M)}$ be a formally self-adjoint semiclassical pseudodifferential operator of order $m$ with principal symbol $p \in S^m(T^*M)$. Let $E \in \R$ and assume that $p^{-1}(E-\delta_E,E+\delta_E)$ is compact for some $\delta_E > 0$, and $\partial_\xi p \neq 0$ on $p^{-1}(E-\delta_E,E+\delta_E)$. We are interested in the $L^\infty$ norm of eigenfunctions
\[
P_h u_h = (E+ o(1)) u_h, \qquad \|u_h\|_{L^2(M)} = 1,
\]
as $h \to 0$.

\begin{proposition}
\label{proposition:bound}
Let $0 \leq \eps < 1/2$ and $N > 0$. Under the above assumptions, there exists a constant $C := C(E,N,\eps) > 0$ such that for all $\delta \geq h^{\eps}$, for all $x \in M$:
\[
|u_h(x)|^2 \leq C\left(\left(\delta^{-1}  h^{-(n-1)}  + \delta h^{-(n-1)-2\eps}\right) \int_{B(x,C\delta)} |u_h(y)|^2 \dd y + h^N\|u\|^2_{L^2}\right) 
\]
\end{proposition}

Here $\dd y$ denotes the Riemannian volume on $M$. In the case of the Laplacian, a similar (but sharper) statement can be found in \cite{Donnelly-01}. {\color{black}We believe that the above bound is not sharp and that the term $\delta h^{-(n-1)-2\eps}$ could be removed (see the comments in the proof below where this loss occurs). However, as it suffices for our purposes, we did not push the investigation further.} Letting $\delta=h^\eps$ in the previous proposition, we find:
\begin{equation}
\label{equation:bound-linfty}
|u_h(x)|^2 \leq C \left(h^{-(n-1)-\eps} \int_{B(x,C\delta)} |u_h(y)|^2\dd y + h^{N}\|u_h\|^2_{L^2(M)}\right).
\end{equation}
This bound will be used in the next section to establish our main theorem in the critical energy regime.

\begin{proof}[Proof]
Let $E_h := E+o(1)$. Fix $0 \leq \eps < 1/2$. Let $x_0 \in M$ be an arbitrary point and $N$ an embedded hypersurface passing through $x_0$. Let $\chi_{\delta}(y) = \chi(d(x_0,y)/\delta)$ where $\chi \in C^\infty_{\comp}(\R)$ is a nonnegative bump function such that $\chi \equiv 1$ on $(-1/2,1/2)$ and $0$ outside of $(-1,1)$, $\delta \geq h^{\eps}$. In what follows, we see $\chi_\delta$ as an exotic symbol in the class $S^0_\eps(T^*M)$, {\color{black}see \cite[Page 551]{Dyatlov-Zworski-19} for a definition. We let $\Psi^0_{h,\eps}(M)$ be the corresponding pseudodifferential class of operators.} \\

\emph{Step 1. Bounding the $L^\infty$ norm by the $L^2$ norm on hypersurfaces.} Let $K \subset T^*M$ be a compact subset. First, we claim that for all functions $f_h$ such that $\WF_h(f_h) \subset K$ and $\|f_h\|_{L^2} \leq 1$, one has:
\begin{equation}
\label{equation:bound-hypersurface}
|f_h(x_0)|^2 \leq C h^{-(n-1)} \|A\chi_\delta f_h\|^2_{L^2(N)} + \mc{O}(h^\infty),
\end{equation}
where $A = \Op_h(a)$ and $a \in C^\infty_{\comp}(T^*M)$ is a function with compact support, equal to $1$ on an open neighborhood of $K$. Indeed, as $f_h$ has a compact wavefront set contained in $K$, we can write $\chi_\delta f_h = A \chi_\delta f_h + \mc{O}_{C^\infty}(h^\infty)$. We then consider local coordinates $x = (x_1,x') \in \R^n$ in which $x_0$ is mapped to $0$, and $N$ is mapped to $\{x_1=0\}$. Let $b := b(x',\xi') \in C^\infty_{\comp}(T^*\R^{n-1})$ be a smooth compactly supported function equal to $1$ on large ball around $0$. As the supports of $1-b$ and $a|_{T^*\R^{n-1}}$ are disjoint, one finds by a standard non-stationary phase argument that
\begin{equation}
\label{equation:first}
(\mathbbm{1}-\Op_h^{\R^{n-1}}(b)) A \chi_\delta f_h(x_1=0,\bullet) = \mc{O}_{C^\infty(\R^{n-1})}(h^\infty).
\end{equation}

To prove \eqref{equation:bound-hypersurface}, it then suffices to combine \eqref{equation:first} with the following standard bound applied with $k=n-1$: if $R := \Op^{\R^k}_h(r)$ with $r \in C^\infty_{\comp}(T^*\R^k)$ compactly supported, then there exists a constant $C > 0$ such that $\|R\|^2_{L^2 \to L^\infty} \leq C h^{-k}$. \\

\emph{Step 2. Bounding the $L^2$ norm on hypersurfaces.} It now remains to bound the $L^2$ norm on hypersurfaces appearing on the right-hand side of \eqref{equation:bound-hypersurface}. For that, we use a propagation of singularities type of argument. We let $\Phi_t : T^*M\to T^*M$ be the bicharacteristic flow of $p$, and $H_p \in \mathcal{C}^{\infty}(T^*M,T T^*M)$ be its infinitesimal generator.

Let $(x_0, \xi_0) \in \{p=E_h\} \cap T^*_{x_0}M$. Since $\partial_\xi p \neq 0$, the bicharacteristic $t \mapsto \Phi_t(x_0,\xi_0)$ projects to a non-degenerate curve $t \mapsto \gamma(t) = \pi(\Phi_t(x_0,\xi_0))$ on $M$, where $\pi : T^*M \to M$ is the footpoint projection. Let $X$ be a vector field defined in a neighborhood of $\gamma$ such that $X = d\pi(H_p)$ on $\gamma$. By construction, $X$ is tangent to $\gamma$ and almost tangent to the projected bicharacteristics emanating from points $(x,\xi)$ close to $(x_0,\xi_0)$. We let $N \subset M$ be a hypersurface passing through $x_0$ and transverse to $X$, $V :=  \cup_{t \in (-{\color{black}\rho},{\color{black}\rho})} \varphi_t(N)$ where $(\varphi_t)_{t \in \R}$ denotes the flow generated by $X$, and $U$ be a small open subset in $T^*M$ containing $(x_0,\xi_0)$ projecting onto $V$.

We can then cover a neighborhood of $\{p=E_h\}$ around $x_0$ by open sets $U_j$ corresponding to a point $(x_0,\xi_j)$ as defined above. Let $\sum_j a_j = \mathbf{1}$ be a corresponding partition of unity in $T^*M$ (near $x_0$) and $N_j$ such that on each $\mathrm{supp}(a_j)$, the bicharacteristics of $p$ project to curves which intersect transversally $N_j$ and are almost tangent to the flowlines of a vector field $X_j$, $V_j = \cup_{t \in (-{\color{black}\rho_j},{\color{black}\rho_j})} \varphi_t^j(N_j)$. By \eqref{equation:bound-hypersurface}, we have:
\[
\begin{split}
|u_h(x_0)| & \leq \sum_j |\Op_h(a_j) u_h (x_0)| + \mc{O}(h^\infty) \\
& \leq C h^{-(n-1)/2} \sum_j \|A \chi_\delta \Op_h(a_j) u_h\|_{L^2(N_j)} + \mc{O}(h^\infty),
\end{split}
\]
{\color{black}where the $\mc{O}(h^\infty)$ term in the first inequality is due to the fact that the cover is only near $x_0$, and the operator $A$ is microlocalized near the compact energy shell $p^{-1}(E)$. Note that, we may further assume that $A \equiv \mathbf{1}$ microlocally on the microsupport of the functions $a_j$'s. Therefore, by separation of the microsupports, $(\mathbf{1}-A)\chi_\delta \Op_h(a_j) = \mc{O}_{\Psi^{-\infty}_h}(h^\infty)$ is smoothing and thus:}
\begin{equation}
\label{equation:decomposition}
|u_h(x_0)| \leq C h^{-(n-1)/2} \sum_j \|\chi_\delta \Op_h(a_j) u_h\|_{L^2(N_j)} + \mc{O}(h^\infty).
\end{equation}
Hence, it suffices to bound individually each term above in the sum to obtain the claimed statement.

Let $q_j(x,\xi) := \langle\xi,X_j(x)\rangle$ be the principal symbol of $-ihX_j$. Consider flowbox coordinates on each $V_j$ in which the vector $X_j$ is mapped to $\partial_{x_1}$ and $N_j$ is mapped to $\{x_1=0\}$. Let $\xi = (\xi_1,\xi')$ be the corresponding dual variables. {\color{black}In these coordinates, $q_j(x,\xi) = \xi_1$.} Note that $dp(\partial_{\xi_1}) \neq 0$, otherwise the bicharacteristics of the Hamiltonian $p$ would project to curves tangent to $N_j$ (by construction, they are all transverse to $N_j$ on $U_j$). As a consequence, on $U_j \subset T^*V_j$, we can write $\{p=E_h\} = \{\xi_1=b(x_1,x',\xi')\}$ for some smooth symbol $b$ independent of the $\xi_1$ coordinate. We thus have the factorization $p-E_h = \tilde{p} \cdot (q_j-b) {\color{black}= \tilde{p}\cdot(\xi_1-b)}$ on $U_j$ for some non-vanishing smooth symbol $\tilde{p}$. We therefore obtain:
\begin{equation}
\label{equation:syndic}
\chi_\delta a_j \tilde{p}^{-1} (p-E_h) = (\xi_1-b) \chi_\delta a_j.
\end{equation}

Now, we claim that, uniformly in $t \in (-\rho_j,\rho_j)$:
\begin{equation}
\label{equation:propagation}
\|\chi_\delta \Op_h(a_j) u_h\|_{L^2(N_j)} \lesssim \|\chi_\delta \Op_h(a_j) u_h\|_{L^2(\varphi_t(N_j))}+ {\color{black}|t|^{1/2} h^{-\eps} \|\widetilde{\chi}_\delta u_h\|_{L^2}},
\end{equation}
where $\widetilde{\chi}_\delta$ is a slightly larger bump function supported on $B(x_0,C\delta)$. {\color{black}Let us establish \eqref{equation:propagation}. Writing $B:=\Op^{\R^n}_h(b)$, where $\Op^{\R^n}_h$ denotes the standard (left) quantization in $\R^n$, using that $\chi_\delta$ is a symbol in $S^{0}_{\eps}(T^*\R^n)$, we find by quantizing \eqref{equation:syndic}:
\[
C_h (P_h-E_h) = (-ihX_j - B) \chi_\delta \Op_h(a_j) - R_h,
\]
where $R_h \in h^{1-\eps} \Psi^{\comp}_\eps(\R^n)$, and $C_h$ is a pseudodifferential operator with principal symbol $\chi_\delta a_j \tilde{p}^{-1}$. Note that the remainder is in $h^{1-\eps} \Psi^{\comp}_\eps(\R^n)$, which is better than the expected $h^{1-2\eps} \Psi^{\comp}_\eps(\R^n)$ given by composition of exotic pseudodifferential operators, because it is the product of a standard pseudodifferential operator with an exotic one (and not the product of two exotic ones). We emphasize that it is at this stage that we lose some exponent. Therefore
\[
(-ihX_j - B) \chi_\delta \Op_h(a_j) u_h = R_h u_h =: r_h.
\]
Note that $\WF_h(R_h) \subset \{\chi_\delta \neq 0\}$. Hence, using that $\widetilde{\chi}_\delta$ is a bump function such that $\widetilde{\chi}_\delta \equiv 1$ on $\{\chi_\delta \neq 0\}$ and supported in $B(x_0,C\delta)$ for some some $C > 0$, we find that
\begin{equation}
\label{equation:matin}
r_h = R_h \widetilde{\chi}_\delta u_h + \mc{O}(h^\infty).
\end{equation}}

%
%
%
%
%
%
%

Since $X_j=\partial_{x_1}$ and $b = b(x_1,x',\xi')$, we have that $\Op_h^{\R^n}(b) = \Op_h^{\R^{n-1}}(b(x_1,\bullet,\bullet)) =: B(x_1)$, where $x_1$ should be thought of as a time parameter. From now on, we will therefore write $t=x_1$. For a fixed $t$, $B(t)$ is a (compactly supported) semiclassical pseudodifferential operator acting on $\R^{n-1}$ and the family $t \mapsto B(t)$ is smooth. Hence, the propagator $(t,s) \mapsto U(t,s)$ of the time-dependent family of operators $t \mapsto B(t)$, which satisfies
\[
-ih \partial_{t} U(t,s) = B(t)U(t,s), \qquad U(t,t)=\mathbf{1}
\]
is well-defined (see \cite[Theorem 6.2.3]{Lefeuvre-book} and the remark below for instance). Notice that $(t,s) \mapsto U(t,s) \in \mc{L}(L^2(\R^{n-1}))$ is uniformly bounded (as the principal symbol of $B(t)$ is real-valued) and invertible. Let $v_h := \chi_\delta \Op_h(a_j) u_h$. By definition, $(-ih\partial_t - B(t))v_h(t) = r_h(t)$. The bound \eqref{equation:propagation} is equivalent to $\|v_h(0)\|_{L^2(\R^{n-1})} \lesssim \|v_h(t)\|_{L^2(\R^{n-1})} +|t|^{1/2} h^{-\eps} \|\widetilde{\chi}_\delta u_h\|_{L^2})$.

Using Duhamel's formula, we obtain:
\[
v_h(t) = U(t,0)v_h(0) + \dfrac{i}{h} \int_0^t U(t,s) r_h(s) \dd s.
\]
This yields:
\[
\begin{split}
\|v_h(0)\|_{L^2(\R^{n-1})} & = \|U(t,0)^{-1} U(t,0) v_h(0)\|_{L^2(\R^{n-1})} \lesssim \|U(t,0)v_h(0)\|_{L^2(\R^{n-1})} \\
&\lesssim \|v_h(t)\|_{L^2(\R^{n-1})} + h^{-1} \int_0^t \|U(t,s) r_h(s)\|_{L^2(\R^{n-1})} \dd s \\
& \lesssim \|v_h(t)\|_{L^2(\R^{n-1})} + h^{-1} \int_0^t \|r_h(s)\|_{L^2(\R^{n-1})} \dd s,
\end{split}
\]
where the last inequality follows from the uniform boundedness of the propagator. By Cauchy-Schwarz, the last term is bounded, using \eqref{equation:matin}, by 
\[
\begin{split}
h^{-1}|t|^{1/2} \left(\int_0^t \|r_h(s)\|^2_{L^2(\R^{n-1})} \dd s\right)^{1/2} & \lesssim h^{-1}|t|^{1/2} \|r_h\|_{L^2(\R^{n-1} \times [0,t])} \lesssim |t|^{1/2}h^{-\eps} \|\widetilde{\chi}_\delta u_h\|_{L^2}.
\end{split}
\]
This proves \eqref{equation:propagation}.

We then integrate \eqref{equation:propagation} for $t$ between $(-\delta,\delta)$ and obtain:
\[
\begin{split}
2\delta \|\chi_\delta \Op_h(a_j) u_h\|_{L^2(N_j)}^2 & \leq \int_{-\delta}^{+\delta} \|\chi_\delta \Op_h(a_j) u_h\|_{L^2(\varphi_t(N_j))}^2 \dd t + \delta^{2}h^{-2\eps}\|\widetilde{\chi}_\delta u_h\|^2_{L^2} \\
& \lesssim \|\chi_\delta \Op_h(a_j)u_h\|^2_{L^2(M)} +\delta^{2} h^{-2\eps}\|\widetilde{\chi}_\delta u_h\|^2_{L^2} \\
& \leq (1+\delta^{2} h^{-2\eps})\|\widetilde{\chi}_\delta u_h\|^2_{L^2(M)} + \mc{O}(h^\infty),\\
\end{split}
\]
where we have used in the last line that $\widetilde{\chi}_\delta \equiv 1$ on the support of $\chi_\delta$ (this accounts for the $\mc{O}(h^\infty)$ remainder). Inserting this estimate in \eqref{equation:decomposition}, we find
\[
|u_h(x)|^2 \lesssim (\delta^{-1}  h^{-(n-1)}  + \delta h^{-(n-1)-2\eps})\|\widetilde{\chi}_\delta u_h\|^2_{L^2(M)} + h^N\|u_h\|^2_{L^2(M)},
\]
where the last term comes from the $\mc{O}(h^\infty)$ remainders. This proves the claim.
\end{proof}

\section{Proof of main results}


\label{section:proof}

\subsection{Saturation of Hörmander's bound} We begin by proving Theorem \ref{theorem:magneticzonalstates}. At energy $E=0$, the Hörmander bound is saturated by coherent states at the ground state, namely holomorphic sections $H^0(\Sigma,L^{\otimes k})$, see \cite[Proposition 4.5]{Charles-Lefeuvre-25}. They can be constructed at any point $p_0 \in \Sigma$ using the Bergman kernel. Hence, throughout, we will further assume that $0 < E < E_c$.

We now follow some ideas of the proof of \cite[Proposition 4.6]{Charles-Lefeuvre-25} to construct Gaussian beams and then the magnetic zonal states. Let $(m_k)_{k \geq 0} \geq 0$ be a sequence such that $k^{-2} \lambda_{k,m_k} \to_{k \to +\infty} E$. For simplicity, we drop the dependence of $m_k$ on $k$ in the following and simply write $m$. Recall that $a$ denotes the principal symbol of the family of twisted semiclassical pseudodifferential operators $(\mathbf{A}_k)_{k \geq 0}$, see \S\ref{ssection:averaging} (and Remark~\ref{remark:cutoff} to justify that we may regard this family as genuine pseudodifferential operators on the energy shells $E < E_c$).

{\color{black}Let $p_0 \in \Sigma$. Using \eqref{equation:lambdamk} and the relation $p = \beta \circ a$, we find that
\begin{equation}
\label{equation:ck}
T_{p_0}^*\Sigma \cap a^{-1}(m/k) = C(p_0,E_k) = T_{p_0}^*\Sigma \cap p^{-1}(E_k)
\end{equation}
where $E_k = \tfrac{m}{k}(B-\tfrac{m}{2k}) \to_{k \to +\infty} E$. The set $C(p_0,E_k)$ is diffeomorphic to a circle for $E > 0$.}

\subsubsection{Adapted symplectic coordinates} In the following, we drop the dependence on $p_0$ and $E$ in the notation of $\T^2$ introduced in \eqref{equation:torus}. {\color{black}The circle $C(p_0,E_k)$ defined in \eqref{equation:ck} is simply denoted $C$; note that it also depends on $k$ but this is not emphasized.} Fix $z_0 \in C$. We begin by finding suitable local coordinates around $z_0$:

\begin{lemma}
\label{lemma:normal}
    For all $z_0 \in C$, there exists {\color{black}$\eps > 0$}, a neighborhood ${\color{black}\tilde{U}_0}$ of $z_0$ in $T^*\Sigma$ and a symplectic diffeomorphism $\kappa : ({\color{black}\tilde{U}_0},\Omega) \to (V,\omega)$, where $V \subset \R^4$ is an open set, $\omega$ is the canonical $2$-form on $\R^4$, such that:
    \begin{enumerate}[label=\emph{(\roman*)}]
        \item $\kappa(z_0) = 0$;
        \item $\kappa({\color{black}C \cap \tilde{U}_0}) \subset \{(0,0,0)\}\times (-\eps,\eps)$;
        \item Letting $z : (-\eps,\eps) \to C$ be the restriction of $\kappa^{-1}$ to $\{(0,0,0)\}\times (-\eps,\eps)$, we have that $z'(\xi_2) = \partial_\theta$ is the generator of the $2\pi$-periodic rotation in $C$;
        \item The symplectomorphism $\kappa$ trivializes the symbol $a$ in the sense that
    \[a(\kappa^{-1}(x_1,x_2,\xi_1,\xi_2)) - m / k = \xi_1.\]
    \end{enumerate}
\end{lemma}

{\color{black}Note that $\kappa$ depends on $k$ but this dependence is uniform as $k \to +\infty$, that is all its derivatives are uniformly bounded. For simplicity, this dependence is not emphasized in the notations.}

\begin{proof}
    Observe that $C$ is transverse to the Hamiltonian flow of $a$ with respect to the magnetic $2$-form $\Omega$. Let $y^2 \in (-\eps,\eps)\subset \R \mapsto z(y^2) \in C$ be a local coordinate for $C$ around $z_0$ such that $z(0) = z_0$, and let
    \[
    \Lambda := \{\Phi_{y^1}^{{\color{black}a}}(z(y^2)) ~:~ (y^1,y^2) \in U \subset \R^2\}
    \]
    be the flow-out of $C$ with respect to $(\Phi_{y^1}^{a})_{y^1 \in (-\eps,\eps)}$, the Hamiltonian flow of $a$ with respect to $\Omega$. Here, $U$ is a neighborhood of $(0,0)$ in $\R^2$ and a local coordinate domain for $\Lambda$ around $z_0$. Note that $\Lambda$ is a Lagrangian submanifold of $(T^*\Sigma,\Omega)$ (Lemma \ref{lemma:lagrangian-torus}). Thus, by the Weinstein tubular theorem, there exists a local symplectic diffeomorphism 
    \[
    \phi : (\mc{U},\omega_{\mathrm{Liouv}}) \subset T^*\Lambda \to ({\color{black}\tilde{U}_0},\Omega) \subset T^*\Sigma,
    \]
    where $\omega_{\mathrm{Liouv}}$ is the canonical Liouville $2$-form on $T^*\Lambda$, {\color{black}$\mc{U}$ is a neighborhood of the zero section of $T^*\Lambda$}, and $\tilde{U}_0$ is a neighborhood of $z_0$ in $T^*\Sigma$; in addition, after identification of $\Lambda$ with the zero section of $T^*\Lambda$, $\phi_{|\Lambda}$ is the identity.
    
    We now identify $(\mc{U},\omega_{\mathrm{Liouv}})$ to a neighborhood of $0$ in $(\R^4,\omega_{\R^4})$, with coordinates $(y^1,y^2,\eta_1,\eta_2)$, where $\omega_{\R^4} = d\eta \wedge dy$ is the canonical $2$-form on $\R^4$. In these coordinates, $C\cap \mc{U} = \{y^1 = \eta_1 = \eta_2 = 0\}$ and $\Lambda = \{\eta_1=\eta_2=0\}$. If we define $\tilde{a}(y^1,y^2, \eta_1,\eta_2) := a(\phi(y^1,y^2,\eta_1,\eta_2)) - m/k$, then $\tilde{a}$ vanishes on (the image of) $\Lambda$ in $\R^4$ (hence on the image of $C$). In addition,
    \[
    \Phi_t^{\tilde{a}}(y^1,y^2,0,0) = (y^1 + t, y^2, 0,0)
    \]
    where $(\Phi_t^{\tilde{a}})_{t \in \R}$ denotes the Hamiltonian flow of $\tilde{a}$ computed with respect to the canonical form $\omega_{\R^4}$. We then consider the linear symplectic change of variables $(q^1,q^2,p_1,p_2) := (y^1, -\eta_2, \eta_1, y^2)$. In these coordinates, $C\cap \mc{U}$ is a subset of $\{(0,0,0)\}\times \R$, and
    \begin{equation}
        \label{equation:dphit}
    \Phi_t^{\tilde{a}}(q^1,0,0,p_2) = (q^1 + t, 0, 0,p_2).
    \end{equation}

    
    
    Let $b$ be a smooth function defined near $0$ in $\R^4$, equal to $b=p_2$ on $\{0\} \times \R^3$, and such that $H_{\tilde{a}} b = 0$ (such a function exists since, by \eqref{equation:dphit}, $H_{\tilde{a}}(0) = \partial_{q^1}$, and this vector field is transverse to $\{0\} \times \R^3$). Note that $H_b(0) = \partial_{q_2}$ is transverse to $H_{\tilde{a}}(0)$ at $0$ (hence locally around $0$). For any $(q,p)$ near $0$, there exists unique times $t_1,t_2 \in \R$ and unique points $(0,0,p_1',p_2') \in \{(0,0)\}\times \R^2$ such that $(q,p) = \Phi_{t_1}^{\tilde{a}} \circ \Phi_{t_2}^{b} (0,0,p_1',p_2')$. By construction, $dt_1(H_{\tilde{a}})=1=dt_2(H_b)$ and $dt_1(H_b)=0=dt_2(H_{\tilde{a}})$. This easily implies that $(q,p) \mapsto (x_1,x_2,\xi_1,\xi_2) := (t_1(q,p),t_2(q,p),\tilde{a}(q,p),b(q,p))$ is a symplectic change of variables fixing $\{(0,0,0)\}\times \R$; in these coordinates, $\tilde{a}$ is mapped to $\xi_1$.
    
    Finally, to satisfy item (iii), it suffices to postcompose the map constructed here by a symplectic diffeomorphism $(x_1,x_2,\xi_1,\xi_2) \mapsto (x_1,x_2/\psi'(\xi_2), \xi_1, \psi(\xi_2))$, where $\psi : \R \to \R$ is a well-chosen diffeomorphism making the norm of $z'$ constant equal to $1$.
\end{proof}

\subsubsection{Construction of magnetic zonal modes}  Let $U \subset \Sigma$ be a small neighborhood of $p_0$, and $s \in C^{\infty}(U,L)$ be a local section such that $|s| = 1$ pointwise on $U$. As $(\mathbf{A}_k)_{k \geq 0}$ is a twisted family of pseudodifferential operators (see \cite[Lemma 4.1]{Charles-Lefeuvre-25}), there exists a family of standard semiclassical pseudodifferential operators $A_{1/k} \in \Psi_{1/k}(U)$ (with semiclassical parameter $h := 1/k$) such that for all $f \in C^{\infty}_{\comp} (U)$, one has
\[
A_{1/k} f = \overline{s}^{\otimes k} \mathbf{A}_k (f s^{\otimes k}),
\]
where the equality holds on $U$, see \S\ref{ssection:twisted-pseudos}. {\color{black}Note that, strictly speaking, spectral cutoffs should be inserted around $\mathbf{A}_k$, but we omit them to avoid burdening the notation (see Remark~\ref{remark:cutoff}).} Writing locally $\nabla s = -i\beta \otimes s$ for some $1$-form $\beta \in C^{\infty}(U,T^*U)$, we have that the principal symbol of $A_{1/k}$ is given by
\[
\sigma_{A_{1/k}}(x,\xi) = \sigma_{\mathbf{A}}(x, \xi -\beta(x) ; k),
\]
see \eqref{equation:psymbol}. We may rewrite this as $\sigma_{A_{1/k}} = S^*  \sigma_{\mathbf{A}}$ using the magnetic shift
\[
S(x,\xi) := (x,\xi - \beta(x)).
\]

Recall that the twisted symplectic $2$-form is defined as $\Omega = \omega_0 + i\pi^*F_{\nabla}$, where $\omega_0$ is the Liouville $2$-form on $T^*\Sigma$, $\pi : T^*\Sigma \to \Sigma$ is the projection and $F_{\nabla}$ is the curvature of the line bundle $L \to \Sigma$. The trivializing section $s$ allows to identify the connection $\nabla$ on $L$ with $d-i\beta$; thus $F_\nabla = -id\beta$ on $U$ and $\Omega = \omega_0 + \pi^*d\beta$. Observe that $S^* \Omega = \omega_0$. Hence, setting $K := \kappa \circ S$ where $\kappa$ was introduced in Lemma \ref{lemma:normal}, we find that $K$ is defined on a neighborhood {\color{black}$\tilde{U} := S^{-1}(\tilde{U}_0)$} of $S^{-1}(z_0) = z_0 + \beta(\pi(z_0)) =: Z_0 \in T_{p_0}^*\Sigma$, and is a symplectic diffeomorphism
\[
K : (\tilde{U},\omega_0) \to (V,\omega_{\R^4})
\]
where $V \subset \R^4$ is an open set and $\omega_{\R^4}$ is the canonical $2$-form of $\R^4$. In addition, $K(Z_0) = 0$, $K(S^{-1}(C)) \subset \{(0,0,0,\R)\}$, and 
\[
\sigma_{A_{1/k}}(K^{-1}(x_1,x_2,\xi_1,\xi_2)) - m/k = \xi_1.
\]
In particular, we may parameterize $C$ locally around $p_0 = S(K^{-1}(0))$ by
\[
C = \{z(\xi_2) := S\circ K^{-1}(0,0,0,\xi_2)) ~:~ \xi_2 \in (-\eps,\eps)\},
\]
for some $\eps > 0$, and $z'(\xi_2) = \partial_\theta$.

We then quantize $K$ locally around $Z_0$ by a unitary semiclassical Fourier Integral Operator $T$ such that $T$ microlocally conjugates $A_{1/k} - m/k$ to $k^{-1}D_{x_1}$. That is $T : L^2(U) \to L^2(\R^2)$ where $U \subset \Sigma$ is the open subset defined above in a neighborhood of $p_0$ and
\begin{equation}
\label{equation:tt}
T^{-1} \circ k^{-1}D_{x_1} \circ T \equiv A_{1/k} - m/k, \qquad T^*T \equiv \mathbf{1},
\end{equation}
{\color{black}where both equalities hold microlocally on $\tilde{U}$, the neighborhood introduced above,} and $\equiv$ stands for equality modulo $\mc{O}_{L^2}(k^{-\infty})$. From now on, we will use $h=1/k$ as a semiclassical parameter.

For later use, it will be useful to recall that the Schwartz kernel of $T^{-1}$ has the following form:

\begin{lemma}
\label{lemma:calor}
In a neighborhood of $(p_0,0) \in \Sigma \times \R^2$, one has:
\begin{equation}
    \label{equation:schwartz}
T^{-1}(p,x) = \dfrac{1}{(2\pi h)^{2}}\int_{\R^2} e^{\frac{i}{h} (\phi(p,\eta)-x\cdot\eta)} q(p,\eta ; h) \dd \eta,
\end{equation}
where $\phi$ and $q$ satisfy the following properties:
\begin{enumerate}[label=\emph{(\roman*)}]
    \item $\phi$ is a non-degenerate phase function such that $K^{-1}(x,\eta) = (p,\xi)$ if and only if $\xi = \partial_p \phi(p,\eta)$ and $x = \partial_\eta \phi(p,\eta)$; the nondegeneracy of $\phi$ means that
    \begin{equation}
    \label{equation:two}
    \det(\partial^2_{p,\eta}\phi) \neq 0,
\end{equation}
for all $(p,\eta)$ in the support of $q$.
\item Additionally, for all $\xi_2$ near $0$, one has 
\begin{equation}
    \label{equation:calculs}
    \partial_\eta\phi(p_0,(0,\xi_2)) = 0.
\end{equation}

    \item $q = q_0 + \mc{O}(h)$, where $q_0$ is a {\color{black}real-valued} symbol of order zero which is, uniformly in $k$, compactly supported in $\eta$, and {\color{black}$q_0 > 0$} near $(p_0,0)$.
\end{enumerate} 
\end{lemma}

{\color{black}Although not emphasized in the notations, the phase $\phi$ also depends on $h$ as the map $K$ does but this dependence is uniform.} Note that the canonical relation of $T^{-1}$ is given locally by
\begin{equation}
    \label{equation:canonical-relation}
    \mc{C} = \{(K^{-1}(x,\eta), (x,\eta)) ~:~ (x,\eta) \in V\}.
\end{equation}

\begin{proof}
Items (i) and (iii) are standard, see \cite[Chapter 25, \S3]{Hormander-4} or \cite[Chapter 10, \S2]{Zworski-12} for instance. To prove \eqref{equation:calculs}, observe that $(p,\eta) \mapsto (\partial_\eta \phi(p,\eta), \eta)$ is a (local) diffeomorphism near $(p_0,0)$ as $\det(\partial^2_{p,\eta}{\color{black} \phi}) \neq 0$. Hence, for all $(0,\xi_2) \in \R^2$ near $0$, there is a unique $p_\star$ such that $(0,(0,\xi_2)) = (\partial_\eta\phi(p_\star,(0,\xi_2)), (0,\xi_2))$. Hence, $\partial_\eta\phi(p_\star,(0,\xi_2)) = 0$. As $K^{-1}(0,(0,\xi_2)) = K^{-1}(\partial_\eta\phi,\eta) = (p_\star, \partial_p \phi(p_\star, (0,\xi_2)))$ and, by construction, the image by $K^{-1}$ of $(0,(0,\xi_2))$ is contained in the fiber $T^*_{p_0}\Sigma$ above $p_0$, one has $p_\star = p_0$. This proves \eqref{equation:calculs}.
\end{proof}


For $\xi \in \R^2$ in a neighborhood of $0$, we now define the Gaussian coherent state $\Psi_{\xi} \in C^{\infty}_{\comp}(\R^2)$ by
\begin{equation}
    \label{equation:gaussian-states}
\Psi_{\xi}(x) := (2 \pi h)^{-1/2} e^{-\frac{1}{2h} |x|^2 + \frac{i}{h} x \cdot \xi} \phi(x),
\end{equation}
where $\phi \in C_{\comp}^{\infty}(\R^2)$ is a bump function localized around $0$ such that $\phi \equiv 1$ in a neighborhood of $0$. {\color{black}Note that this is not an $L^2$ normalized state (it has $L^2$ norm $1/\sqrt{2}$); this choice will be convenient later.} For $z = z(\xi_2) \in C$ close to $z_0$, we set
\[
e_{k,m,z(\xi_2)} := T^{-1} \Psi_{0,\xi_2} \in C^{\infty}(\Sigma,\C),
\]
and consider the Gaussian coherent state
\[
\mathbf{e}_{k,m,z} := e_{k,m,z} s^{\otimes k} \in C^{\infty}(\Sigma, L^{\otimes k}).
\]
As in \cite[Proposition 4.6]{Charles-Lefeuvre-25}, we introduce the \textit{Gaussian beam}
\begin{equation}
    \label{equation:gaussian-beams}
\mathbf{f}_{k,m,z} := h^{-1/4} \Pi_{k,m} \mathbf{e}_{k,m,z}
\end{equation}
associated with $\mathbf{e}_{k,m,z}$. We emphasize that $\mathbf{f}_{k,m,z}$ for $z$ close to $z_0$ \emph{does depend} on a choice of symplectomorphism $K$ and Fourier Integral Operator $T$. {\color{black}As we will see below in \eqref{equation:ffprime}, the sections $\mathbf{f}_{k,m,z}$ are (up to a constant multiple) $L^2$ normalized.}




Finally, we introduce the eigenfunction
\begin{equation}
\label{equation:calor}
\mathbf{u}_{k,m,p_0} := h^{-1/4} \int_{\R} \chi(\xi_2) \mathbf{f}_{k,m,z(\xi_2)} \dd \xi_2,
\end{equation}
where $\chi \in C^\infty_{\mathrm{comp}}(\R)$ is a smooth nonnegative localizer with support in $(-\eps,\eps)$ (where $\eps > 0$ was introduced in Lemma \ref{lemma:normal}). We call $\mathbf{u}_{k,m,p_0} \in C^\infty(\Sigma,L^{\otimes k})$ a \emph{magnetic zonal state}. We will show that it saturates Hörmander's $L^\infty$ bound; however, its defect measure is not yet the Lebesgue measure on the full torus $\T^2(p_0,E)$, see \S\ref{sssection:modif} where this is further discussed. {\color{black}A heuristic for the normalizing factor $h^{-1/4}$ is given in the paragraph following \eqref{equation:magnetic-zonal-state}.}





\subsubsection{Estimating the $L^2$ norm} We begin by studying the $L^2$ norm of $\mathbf{u}_{k,m,p_0}$. Recall that $z : (-\eps,\eps) \to C$ is the local diffeomorphism provided by Lemma \ref{lemma:normal} and induced by the symplectomorphism $K$. We first estimate the following scalar product:

\begin{lemma} For $\xi_2,\xi_2'$ close to $0$, there holds:
\begin{equation}
    \label{equation:ffprime}
\langle\mathbf{f}_{k,m,z(\xi_2)}, \mathbf{f}_{k,m,z(\xi_2')}\rangle_{L^2(\Sigma,L^{\otimes k})} = {\color{black}\dfrac{1}{2\sqrt{\pi}}} e^{-\frac{1}{4h} (\xi_2 - \xi_2')^2} + \mc{O}(h^{\infty}).
\end{equation}
\end{lemma}

\begin{proof}
    Write, with $z := z(\xi_2)$ and $z' := z(\xi_2')$:
    \[
    \begin{split}
h^{1/2} \langle\mathbf{f}_{k,m,z}, \mathbf{f}_{k,m,z'}\rangle_{L^2(\Sigma,L^{\otimes k})} &= \langle\Pi_{k,m} \mathbf{e}_{k,m,z}, \Pi_{k,m}\mathbf{e}_{k,m,z'}\rangle_{L^2(\Sigma,L^{\otimes k})} \\
&= \langle\Pi_{k,m} \mathbf{e}_{k,m,z}, \mathbf{e}_{k,m,z'}\rangle_{L^2(\Sigma,L^{\otimes k})}\\
&= \frac{1}{2\pi} \int_0^{2\pi} \langle e^{itk(\mathbf{A}_k - m/k)} \mathbf{e}_{k,m,z}, \mathbf{e}_{k,m,z'}\rangle_{L^2(\Sigma,L^{\otimes k})} \dd t \\
&= \frac{1}{2\pi} \int_0^{2\pi} \langle e^{itk(A_{1/k} - m/k)} e_{k,m,z}, e_{k,m,z'}\rangle_{L^2(\Sigma)} \dd t,
\end{split}\]
where we have used in the last equality that $|s|=1$ pointwise on $U$.

Since $e^{2i\pi k(\mathbf{A}_k - m/k)} = \mathbf{1}$, we may write 
\[
h^{1/2}\langle\mathbf{f}_{k,m,z}, \mathbf{f}_{k,m,z'}\rangle_{L^2(\Sigma,L^{\otimes k})} = \frac{1}{2\pi} \int_{-\pi}^{\pi} \langle e^{itk(A_{1/k} - m/k)} e_{k,m,z}, e_{k,m,z'}\rangle_{L^2(\Sigma)} \dd t.
\]
Similarly to the proof of \cite[Proposition 4.6]{Charles-Lefeuvre-25}, by elementary wavefront set arguments, the integral is modified by a $\mc{O}(h^{\infty})$ when restricted to a neighbourhood of $t = 0$, say $t\in \mathrm{supp}(\zeta)$ where $ 0\leq \zeta \leq 1$, $\zeta(0) \equiv 1$ on a small neighborhood of $0$ and $\zeta \in C^{\infty}_{\comp}(\R)$. {\color{black}Indeed, this is because the Hamiltonian flow of $\mathbf{A}_k$ is transverse to $C(p_0,E)$, and $\Phi_t(z) \neq z$ for all $0 < t < T_E$. Consequently, for $t$ outside of a small neighborhood of $0$, the evolved wave packet $e^{itk(A_{1/k} - m/k)} e_{k,m,z}$ is microlocalized at a strictly positive distance away from $z'$ (provided $z'$ is close enough to $z$). We also emphasize that, strictly speaking, we should be working with $\mathbf{A}_k$ multiplied by an appropriate smooth microlocal cutoff; as already mentioned in Remark \ref{remark:cutoff}, we do not insist on this point.} In particular, we may choose $\zeta$ such that the microsupports of $e^{itk(A_{1/k} - m/k)} e_{k,m,z}$ and $e_{k,m,z'}$ are inside $U$. Therefore, modulo $\mc{O}(h^{\infty})$, using that $T$ is unitary, we find:
\begin{equation}
    \label{equation:simplification}
\langle e^{itk(A_{1/k} - m/k)} e_{k,m,z}, e_{k,m,z'}\rangle_{L^2(\Sigma)} = \langle e^{t\partial_{x_1}} \Psi_{0,\xi_2}, \Psi_{0,\xi_2'}\rangle_{L^2(\R^2)}.
\end{equation}

Thus, writing $\equiv$ as above for equality modulo $\mc{O}(h^\infty)$, we find:
\[
\begin{split}
    h^{1/2}\langle\mathbf{f}_{k,m,z}&, \mathbf{f}_{k,m,z'}\rangle_{L^2(\Sigma,L^{\otimes k})} \\
    & \equiv (2\pi)^{-1}\int \zeta(t) \langle e^{t\partial_{x_1}} \Psi_{(0,\xi(z))}, \Psi_{(0,\xi(z'))}\rangle_{L^2(\R^2)} \dd t\\
    & \equiv  (2\pi)^{-1}\int_{\R}\int_{\R^2} \Psi_{(0,\xi(z))}(x_1 + t,x_2) \overline{\Psi_{(0,\xi(z'))}(x_1,x_2)} \dd x_1 \dd x_2 \dd t \\
    & \equiv  (2\pi)^{-1}(2\pi h)^{-1}\int_{\R} \int_{\R^2} e^{-\frac{1}{2h} ((x_1 + t)^2 + x_2^2 + x_1^2 + x_2^2) + \frac{i}{h} (\xi_2 - \xi_2') x_2} \dd x_1 \dd x_2 \dd t  \\
    & \equiv  (2\pi)^{-1}(2\pi h)^{-1}\int_{\R} \int_{\R^2} e^{-\frac{t^2}{{\color{black}4}h}} e^{-\frac{1}{h}( (x_1 + t/2)^2 + x_2^2) + \frac{i}{h} (\xi_2 - \xi_2') x_2} \dd x_1 \dd x_2 \dd t  \\
    & \equiv  1/2 \cdot (2\pi)^{-1}\int_{\R} e^{-\frac{t^2}{{\color{black}4}h}} e^{-\frac{1}{4h} (\xi_2 - \xi_2')^2}\dd t  \equiv  {\color{black}\dfrac{h^{1/2}}{2\sqrt{\pi}}}e^{-\frac{1}{4h} (\xi_2 - \xi_2')^2}.
\end{split}
\]
This proves the claim. 
\end{proof}

As a consequence, we obtain:

\begin{lemma}
\label{lemma:l2}
{\color{black}The following holds:
\[ \|\mathbf{u}_{k,m,p_0}\|^2_{L^2(\Sigma,L^{\otimes k})} = \|\chi\|^2_{L^2(\R)} + \mc{O}(h).
\]}
\end{lemma}

\begin{proof}
Using \eqref{equation:ffprime}, we obtain:
\[
\begin{split}
    \|\mathbf{u}_{k,m,p_0}\|_{L^2(\Sigma,L^{\otimes k})}^2 &= h^{-1/2} \iint_{{\color{black}\R \times \R}}\langle\textbf{f}_{k,m,z(\xi_2)}, \textbf{f}_{k,m,z(\xi_2')}\rangle_{L^2(\Sigma,L^{\otimes k})} \chi(\xi_2) \chi(\xi_2') \dd \xi_2 \dd \xi_2' \\
    &= {\color{black}\dfrac{h^{-1/2}}{2\sqrt{\pi}}} \iint_{{\color{black}\R \times \R}} e^{-\tfrac{1}{4h}(\xi_2 - \xi_2')^2} \chi(\xi_2) \chi(\xi_2') \dd \xi_2 \dd \xi_2' + \mc{O}(h^{\infty}).
\end{split}
\]
{\color{black}Making the successive changes of variables $u = \xi_2-\xi_2'$ and then $v=u/(2\sqrt{h})$, we find that
\[
\begin{split}
\iint_{\R \times \R} e^{-\tfrac{1}{4h}(\xi_2 - \xi_2')^2} \chi(\xi_2) \chi(\xi_2') \dd \xi_2 \dd \xi_2'  & = \iint_{\R \times \R} e^{-\tfrac{1}{4h}u^2} \chi(\xi_2) \chi(\xi_2-u) \dd \xi_2 \dd u \\
& = 2 \sqrt{h} \iint_{\R \times \R} e^{-v^2} \chi(\xi_2) \chi(\xi_2-2\sqrt{h}v) \dd \xi_2 \dd v \\
& = 2\sqrt{\pi h} \|\chi\|^2_{L^2(\R)} + \mc{O}(h^{3/2}),
\end{split}
\]
where the remainder term is obtained by a Taylor expansion to order $2$ of $v \mapsto \chi(\xi_2-2\sqrt{h}v)$ at $v=0$, observing that the term of order $1$ does not contribute as it is odd with respect to $v$. This proves the claim.}

\end{proof}

\subsubsection{Estimating the $L^\infty$ norm}

We now estimate the $L^\infty$ norm of $\mathbf{u}_{k,m,p_0}$ at $p_0$. We start by the following:

\begin{lemma}
\label{lemma:f}
For any $z \in C \cap \tilde{U}$, there holds
\[
\mathbf{f}_{k,m,z}(p_0) = h^{-1/4} {\color{black}\dfrac{q_0(p_0,(0,\xi_2))}{2\pi}} s^{\otimes k}(p_0) + \mc{O}(h^{3/4}).
\]
\end{lemma}

\begin{proof}
Write, for $z = z(\xi_2)$,
\[
\begin{split}
h^{1/4}\textbf{f}_{k,m,z}(p_0) &= \frac{1}{2\pi} \int_{-\pi}^\pi \left(e^{itk (\textbf{A}_k - m/k)} \textbf{e}_{k,m,z}\right)(p_0) \dd t \\
&= \left(\frac{1}{2\pi} \int_{-\pi}^\pi \zeta(t) \left(e^{itk (A_{1/k} - m/k)} e_{k,m,z}\right)(p_0) \dd t\right) s^{\otimes k}(p_0) + \mc{O}(h^{\infty})\\
&= \left(T^{-1} \left(\frac{1}{2\pi} \int_{\R} \zeta(t) e^{t\partial_{x_1}} \Psi_{0,\xi_2} \dd t \right)\right)(p_0) s^{\otimes k}(p_0) + \mc{O}(h^{\infty}),
\end{split}\]
{\color{black}where $\zeta$ is a smooth cutoff function equal to $1$ near $0$. As above, this is allowed because the Hamiltonian flow of $\mathbf{A}_k$ is transverse to $C(p_0,E)$ so that, evaluating at $p_0$, only the behaviour for $t$ in an arbitrarily small neighborhood of $0$ is relevant.} The quantity we want to study is thus
\[
h^{1/4} f_{k,m,z}(p_0) := \left(T^{-1} \left(\frac{1}{2\pi} \int_{\R} \zeta(t) e^{t\partial_{x_1}} \Psi_{0,\xi_2} dt \right)\right)(p_0).
\]
To compute it, we use the expression \eqref{equation:schwartz} for the Schwartz kernel of $T^{-1}$. First, observe that for $x = (x_1,x_2)$ close to $0$, we have:
\[
\begin{split}
     \int_{\R} \zeta(t) e^{t\partial_{x_1}} \Psi_{(0,\xi_2)}(x) \dd t &= \int_{\R} \zeta(t) \Psi_{(0,\xi_2)}(x_1 + t,x_2) \dd t \\
     &= (2\pi h)^{-1/2} \int_{\R} \zeta(t) e^{-\frac{1}{2h} ((x_1 + t)^2 + x_2^2)} e^{\frac{i}{h} x_2 \xi_2} \dd t \\
     &= (2\pi h)^{-1/2} e^{-\frac{1}{2h} x_2^2 + \frac{i}{h} x_2 \xi_2}  \int_{\R} \zeta(t)  e^{-\frac{(x_1 + t)^2}{2h}} \dd t \\
     &= e^{-\frac{1}{2h} x_2^2 + \frac{i}{h} x_2 \xi_2} \psi_h(x_1),
\end{split}\]
where $x_1 \mapsto \psi_h(x_1)$ is a smooth function, uniformly bounded in $h$ as well as all its derivatives, such that $\psi_h(0) = 1 + \mc{O}(h^\infty)$. Using \eqref{equation:schwartz}, we may write (recall that $\equiv$ stands for equality modulo $\mc{O}(h^\infty)$):
\[
\begin{split}
h^{1/4} f_{k,m,z} (p_0) & \equiv  \dfrac{1}{(2\pi h)^{2}} \int_{\R^2 \times \R^2} e^{\frac{i}{h}(\phi(p_0,\eta)-x\cdot \eta)} q(p_0,\eta ; h)e^{-\frac{1}{2h} x_2^2 + \frac{i}{h} x_2 \xi_2} \psi_h(x_1) \dd x_1 \dd x_2 \dd \eta \\
& \equiv \dfrac{1}{(2\pi h)^{2}} \int_{\R^2 \times \R^2} e^{\frac{i}{h}\Phi(x,\eta)} q(p_0,\eta ; h) \psi_h(x_1)/2\pi ~\dd x_1 \dd x_2 \dd \eta,
\end{split}
\]
where the complex-valued phase is given by:
\[
\Phi(x,\eta) = \phi(p_0,\eta)-x\cdot\eta + \tfrac{i}{2} x_2^2 + x_2 \xi_2.
\]
The stationary point of the phase verifies
\[
\Im \Phi = 0, \partial_{\eta,x} \Phi = 0 \Longleftrightarrow x_2 = 0, x = \partial_\eta \phi(p_0,\eta), \eta = (0,\xi_2).
\]
Note that, by \eqref{equation:calculs}, $\partial_\eta \phi(p_0,(0,\xi_2)) = x = 0$. Consequently, the unique stationary point is equal to $x=0, \eta = (0,\xi_2)$, and the Hessian at this point is given by
\[
\partial^2 \Phi(x,\eta) = \begin{pmatrix}
\partial_\eta^2 \phi & -\mathbbm{1} \\
-\mathbbm{1} & \begin{pmatrix}
0 & 0 \\
0 & i
\end{pmatrix}
\end{pmatrix}.
\]
{\color{black}By expanding the determinant with respect to the third column, one easily sees that the determinant of a block matrix 
\[
\begin{pmatrix} A & -\mathbbm{1} \\ -\mathbbm{1} & \begin{pmatrix}
0 & 0 \\
0 & i
\end{pmatrix} \end{pmatrix}, \qquad A = \begin{pmatrix} a & b \\ c & d \end{pmatrix}
\]
is $1-id$. Note that, by \eqref{equation:calculs}, $\xi_2 \mapsto \phi(p_0,(0,\xi_2))$ is constant. Hence, differentiating twice with respect to $\xi_2$, we find that $\partial_{\eta_2} \partial_{\eta_2} \phi(p_0,(0,\xi_2))=0$, that is $d=0$. Hence, the determinant is $1$.} The (complex) stationary phase lemma (see \cite[Theorem 2.3]{Melin-Sjostrand-75}) then yields:
\[
\begin{split}
h^{1/4} f_{k,m,z}(p_0) &= e^{\frac{i}{h}\phi(p_0,(0,\xi_2))}\dfrac{q(p_0,(0,\xi_2))}{2\pi} + \mc{O}(h)\\
& = e^{\frac{i}{h}\phi(p_0,(0,\xi_2))}\dfrac{q_0(p_0,(0,\xi_2))}{2\pi} + \mc{O}(h).
\end{split}
\]

By \eqref{equation:calculs}, $\xi_2 \mapsto \phi(p_0,(0,\xi_2))$ is constant and, without loss of generality, up to multiplying $T$ by a constant, we can always assume it is $0$. Hence $h^{1/4} f_{k,m,z}(p_0) = q_0(p_0,(0,\xi_2))/2\pi + \mc{O}(h)$. This concludes the proof.

\end{proof}

{\color{black}\begin{lemma}
\label{lemma:linfty}
There exists a constant $C > 0$ such that for all $\chi \in C^\infty_{\mathrm{comp}}(\R), \chi \geq 0$, with support in $(-\eps,\eps)$: 
\[
|\mathbf{u}_{k,m,p_0}| \geq C\left(h^{-1/2} \int_{\R} \chi(t) \dd t - \|\chi\|_{L^\infty}h^{1/2}\right).
\]
\end{lemma}

\begin{proof}
Combining \eqref{equation:calor} and Lemma \ref{lemma:f}, we find:
\[
\begin{split}
\mathbf{u}_{k,m,p_0}(p_0) &= h^{-1/4} \int_{\R} \chi(\xi_2) \mathbf{f}_{k,m,z(\xi_2)}(p_0) \dd \xi_2 \\
& = h^{-1/2} \int_{\R} \chi(\xi_2)  \dfrac{q(p_0,(0,\xi_2))}{2\pi} \dd \xi_2 s^{\otimes k}(p_0) + \mc{O}(\|\chi\|_{L^\infty} h^{1/2}).
\end{split}
\]
Applying the pointwise norm at $p_0$ in the previous equality, and using that $q _0> C$ near $(p_0,0)$ for some constant $C > 0$, yields the result stated in the lemma.
\end{proof}}

\subsubsection{Proof completion}

We can now complete the proof of Theorem \ref{theorem:magneticzonalstates}:

\begin{proof}[Proof of Theorem \ref{theorem:magneticzonalstates}]
{\color{black}Combining Lemmas \ref{lemma:l2} and Lemma \ref{lemma:linfty}, we see that $\mathbf{u}_{k,m,p_0}$ (defined in \eqref{equation:calor}) saturates the Hörmander bound. More precisely, there exists a constant $C > 0$ such that for all $\chi \in C^\infty_{\mathrm{comp}}(\R)$, with $\mathbf{u}_{k,m,p_0}$ given by \eqref{equation:calor}, one has:
\begin{equation}
\label{equation:lower-bound}
    \dfrac{\|\mathbf{u}_{k,m,p_0}\|_{L^{\infty}(\Sigma,L^{\otimes k})}}{\|\mathbf{u}_{k,m,p_0}\|_{L^2(\Sigma,L^{\otimes k})}} \geq C \dfrac{\int_{\R} \chi(t) \dd t + \mc{O}_\chi(h)}{(\int_{\R} \chi(t)^2 \dd t)^{1/2} + \mc{O}_{\chi}(h^{1/2})}h^{-1/2},
\end{equation}
where $\mc{O}_\chi(h)$ means that this term is bounded by $\leq C_\chi h$ for some constant $C_\chi > 0$ depending on $\chi$.}
\end{proof}

\subsection{Semiclassical defect measure of magnetic zonal states}

{\color{black}\subsubsection{Defect measure}

 From now on, to simplify notation, we drop the dependence in $p_0,m$. Recall that $\mathbf{u}_{k}$ was defined in \eqref{equation:calor}. Up to composing by the local diffeomorphism $z : (-\eps,\eps) \to C$, the function $\chi$ in the definition of $\mathbf{u}_k$ can be seen as a function on $C$. Also recall that the coordinates $(t,\theta)$ on $\T^2(p_0,E)$ were introduced in \S\ref{ssection:lagrangian-tori}.

\begin{lemma}
\label{lemma:scl}
Let $b \in C^\infty_{\comp}(D^*\Sigma)$. Then:
\[
\langle \Op_k(b) \mathbf{u}_{k},\mathbf{u}_{k}\rangle_{L^2} = \dfrac{1}{T_E} \int_{0}^{2\pi} \int_0^{T_E} \chi(\theta)^2 b(t,\theta) \dd t \dd \theta + \mc{O}(h^{1/4}).
\]
\end{lemma}

In particular, Lemma \ref{lemma:scl} shows that the semiclassical limit of $\mathbf{u}_{k}$ is $\chi(\theta)^2 \dd t \dd \theta/T_E$ on $\T^2(p_0,E)$, which proves Theorem \ref{theorem:defect}, item (i). The proof of Lemma \ref{lemma:scl} is based on the following estimate:

\begin{lemma}
For every integer $N > 0$:
\begin{equation}
\label{equation:gaussian}
\begin{split}
\langle \Pi_{k,m} \Op_k(b) \Pi_{k,m} \mathbf{e}_{k,z},\mathbf{e}_{k,z'}\rangle_{L^2}  = \langle b \rangle&(p_0,z) \langle \Pi_{k,m} \mathbf{e}_{k,z}, \mathbf{e}_{k,z'} \rangle_{L^2} \\
&  + \mc{O}(h \langle d(z,z')/\sqrt{h} \rangle^{-N}) + \mc{O}(h^{5/4}).
\end{split}
\end{equation}
\end{lemma}

Here, for $x \in \R$, $\langle x\rangle := (1+x^2)^{1/2}$ denotes the Japanese bracket and $d(z,z')$ is the distance on $C$ between $z$ and $z'$.

\begin{proof}
In \eqref{equation:integral}, we can write the remainder term $h R_h$ where $R_h$ is uniformly bounded on $L^2$. Note that, in \eqref{equation:integral}, we may also freely compose on the right by $\Pi_{k,m}$ and thus change $h R_h$ by $h R_h \Pi_{k,m}$. Consequently, in \eqref{equation:gaussian}, the remainder term of \eqref{equation:integral} yields
\[
|\langle h R_h \Pi_{k,m} \mathbf{e}_{k,z},\mathbf{e}_{k,z'}\rangle_{L^2}| \leq C h \|\Pi_{k,m} \mathbf{e}_{k,z}\|_{L^2}\|\mathbf{e}_{k,z'}\|_{L^2} = \mc{O}(h^{5/4})
\]
by \eqref{equation:ffprime}. Using \eqref{equation:integral}, it thus suffices to prove that
\[
\langle \Op_k(\langle b \rangle) \Pi_{k,m} \mathbf{e}_{k,z},\mathbf{e}_{k,z'}\rangle_{L^2} = \langle b \rangle(p_0,z) \langle \Pi_{k,m} \mathbf{e}_{k,z}, \mathbf{e}_{k,z'} \rangle_{L^2} + \mc{O}(h \langle d(z,z')/\sqrt{h} \rangle^{-N}).
\]
Conjugating by the operator $T$, and using \eqref{equation:tt}, yields
\begin{equation}
\label{equation:lessive}
\begin{split}
\langle \Op_k(\langle b \rangle) \Pi_{k,m} \mathbf{e}_{k,z},\mathbf{e}_{k,z'}\rangle_{L^2} & = \langle T^{-1} (T \Op_k(\langle b \rangle) T^{-1})(T \Pi_{k,m} T^{-1}) \Psi_{(0,\xi_2)}, T^{-1} \Psi_{0,\xi_2'}\rangle_{L^2}  \\
& = \dfrac{1}{2\pi} \int_{\R} \langle  \zeta(t) T \Op_k(\langle b \rangle) T^{-1}  e^{t\partial_{x_1}} \Psi_{(0,\xi_2)} , \Psi_{0,\xi_2'}\rangle_{L^2} \dd t,
\end{split}
\end{equation}
where $z =z(\xi_2), z' = z(\xi_2')$. Note that $e^{t\partial_{x_1}} \Psi_{(0,\xi_2)} = \Psi_{(-t,0;0,\xi_2)}$ (the Gaussian state centered at $x_1=-t,x_2=0$), and $T \Op_k(\langle b \rangle) T^{-1} = \Op_h(p_0 + hp_1+...)$ is a semiclassical pseudodifferential operator in $\R^2$ with principal symbol satisfying $p_0 \circ \kappa = \langle b \rangle$. We then use the standard fact that, for every integer $N > 0$,
\[
\langle \Op_h(a) \Psi_{x,\xi}, \Psi_{x',\xi'} \rangle_{L^2} = a(x,\xi) \langle \Psi_{x,\xi}, \Psi_{x',\xi'} \rangle_{L^2} + \mc{O}(h^{1/2}\langle\|(x,\xi)-(x',\xi')\|/\sqrt{h}\rangle^{-N}).
\]
See \cite[Chapter 2, Theorem 6]{Combescure-Robert-12} for instance. Inserting this estimate applied with $N+1$ in \eqref{equation:lessive}, we find that 
\[
\begin{split}
\left|\langle \Op_k(\langle b \rangle) \Pi_{k,m} \mathbf{e}_{k,z},\mathbf{e}_{k,z'}\rangle_{L^2} \right. &\left.- \langle b \rangle(p_0,z) \langle \Pi_{k,m} \mathbf{e}_{k,z}, \mathbf{e}_{k,z'} \rangle_{L^2}\right| \\
& \leq C \int_{\R} \dfrac{\zeta(t) h^{1/2}}{\left(1+ \tfrac{t^2+(\xi_2-\xi_2')^2}{h}\right)^{(N+1)/2}}\dd t \leq C h \langle (\xi_2-\xi_2')/\sqrt{h} \rangle^{-N}.
\end{split}
\]
This proves the claim.
\end{proof}

\begin{proof}[Proof of Lemma \ref{lemma:scl}]
Let $b \in C^\infty_{\comp}(D^*\Sigma)$. First, note that $\mathbf{u}_k$ is microlocally supported on the energy shell $\{p=E\}$ (it is $\mc{O}_{L^2}(h^{\infty})$ outside). Additionally, observe that for $z,z' \in C$, for every $N > 0$:
\begin{equation}
\label{equation:opkb}
\begin{split}
\langle \Op_k(b) \mathbf{f}_{k,z}&, \mathbf{f}_{k,z'} \rangle_{L^2} \\
&  \overset{\eqref{equation:gaussian-beams}}{=}  h^{-1/2}\langle \Pi_{k,m_k} \Op_k(b) \Pi_{k,m_k} \mathbf{e}_{k,z}, \mathbf{e}_{k,z'}\rangle_{L^2} \\
& \overset{\eqref{equation:gaussian}}{=} h^{-1/2}\left(\langle b\rangle(p_0,z) \langle \Pi_{k,m_k} \mathbf{e}_{k,z}, \mathbf{e}_{k,z'}\rangle_{L^2} +\mc{O}(h \langle (\xi_2-\xi_2')/\sqrt{h} \rangle^{-N}+h^{5/4})\right) \\
& \overset{\eqref{equation:gaussian-beams}}{=} \langle b\rangle(p_0,z) \langle \mathbf{f}_{k,z}, \mathbf{f}_{k,z'}\rangle_{L^2}  +\mc{O}(h^{1/2} \langle (\xi_2-\xi_2')/\sqrt{h} \rangle^{-N}+h^{3/4}).
\end{split}
\end{equation}
We then have:
\[
\begin{split}
& h^{1/2}\langle \Op_k(b) \mathbf{u}_{k},\mathbf{u}_{k}\rangle_{L^2} \\
& \overset{\eqref{equation:calor}}{=}  \int_{\R \times \R} \chi(\xi_2) \chi(\xi_2') \langle \Op_k(b) \mathbf{f}_{k,z(\xi_2)}, \mathbf{f}_{k,z(\xi_2')}\rangle_{L^2} \dd \xi_2 \dd \xi_2' \\
& \overset{\eqref{equation:opkb}}{=} \int_{\R \times \R} \chi(\xi_2) \chi(\xi_2') \left(\langle b\rangle(p_0,z(\xi_2)) \langle \mathbf{f}_{k,z(\xi_2)}, \mathbf{f}_{k,z(\xi_2')} \rangle_{L^2}  + \mc{O}(h^{1/2} \langle (\xi_2-\xi_2')/\sqrt{h} \rangle^{-N}+h^{3/4})\right)\dd \xi_2 \dd \xi_2' \\
& \overset{\eqref{equation:ffprime}}{=} \int_{\R \times \R} \chi(\xi_2) \chi(\xi_2') \langle b\rangle(p_0,z(\xi_2)) \dfrac{1}{2\sqrt{\pi}} e^{-\frac{1}{4h} (\xi_2 - \xi_2')^2}\dd \xi_2 \dd \xi_2' + \mc{O}(h) + \mc{O}(h^{3/4}),
\end{split}
\]
provided $N \geq 2$. As in the proof of Lemma \ref{lemma:l2}, we find that
\[
\int_{\R \times \R} \chi(\xi_2) \chi(\xi_2') \langle b\rangle(p_0,z(\xi_2)) e^{-\frac{1}{4h} (\xi_2 - \xi_2')^2}\dd \xi_2 \dd \xi_2' = 2\sqrt{\pi h} \int_{\R} \chi(\xi_2)^2 \langle b\rangle(p_0,z(\xi_2))  \dd \xi_2  + \mc{O}(h^{3/2}),
\]
Combining the above computations, we obtain
\[
\begin{split}
\langle \Op_k(b) \mathbf{u}_{k},\mathbf{u}_{k}\rangle_{L^2} & = \int_{\R} \chi(\xi_2)^2 \langle b\rangle(p_0,z(\xi_2))  \dd \xi_2 + \mc{O}(h^{1/4}) \\
& = \int_{C} \tilde{\chi}(z)^2 \langle b\rangle(p_0,z)  \dd z + \mc{O}(h^{1/4}),
\end{split}
\]
where $\tilde{\chi} = \chi \circ z^{-1}$.
\end{proof}

\subsubsection{Magnetic zonal states with full support on the Lagrangian torus}

\label{sssection:modif}

Let $\mc{R} \subset C(p_0,E)$ denote the set of points $z \in C(p_0,E)$ such that there exists a time $t \in (0,T_E)$ such that $\Phi_t(z) \in C(p_0,E)$. Observe that it follows from the proof of Proposition \ref{proposition:pushforward} that this set is finite (it corresponds in the universal cover to a trajectory passing through a copy $\gamma.i$ of $i$ for $\gamma \in \Gamma$).

The circle $C(p_0,E)$ can be covered by a finite number $N$ of open subsets $(U_i)_{1 \leq i \leq N}$ as introduced in Lemma \ref{lemma:normal}. Let $J_i := C \cap U_i$; up to shrinking $U_i$, the $J_i$'s can be chosen to be intervals on $C$. For $n \geq 0$, we choose on each interval $J_i$ a nonnegative bump function $\chi_n^{(i)} \in C^\infty_{\mathrm{comp}}(J_i)$, $0 \leq \chi_n^{(i)} \leq 1$, with support in $J_i \setminus \mc{R}$ such that $\chi_n^{(i)}$ and $\chi_n^{(j)}$ have disjoint support whenever $i \neq j$. Let
\[
\Lambda_i := \bigcup_{z \in \supp(\chi_n^{(i)})} \{\Phi_t(z) ~:~0 \leq t \leq T_E\}.
\]
Observe that, by construction:
\begin{equation}
\label{equation:faim}
\Lambda_i \cap \Lambda_j = \emptyset, \qquad \forall i \neq j.
\end{equation}

We let $\chi_n := \sum_{i=1}^N \chi_n^{(i)}$. In addition, as $\mc{R}$ is discrete, these functions can be constructed in such a way that $\chi_n^2(\theta) \dd \theta \to_{n \to +\infty} \dd\theta$ in the weak-$*$ topology. Therefore, $\chi_n(\theta)^2 \dd t \dd\theta  \to  \dd t \dd\theta$ on $\T^2(p_0,E)$ in the weak-$*$ topology. For each $1 \leq i \leq N$, we can construct a zonal state $\mathbf{u}_{k,n}^{(i)}$ with corresponding cutoff function $\chi_n^{(i)}$ as in \eqref{equation:calor}. We then define the magnetic zonal state:
\[
\mathbf{u}_{k,n} := \sum_{i=1}^N \mathbf{u}_{k,n}^{(i)}.
\]
The following holds:

\begin{lemma}
\label{lemma:pleurs}
There exists a constant $C > 0$ such that:
\[
\|\mathbf{u}_{k,n}\|_{L^\infty} \geq C h^{-1/2} \int_C \chi_n(\theta) \dd \theta + \mc{O}_n(h^{1/2}), \qquad \|\mathbf{u}_{k,n}\|^2_{L^2} = \|\chi_n\|^2_{L^2(C)} + \mc{O}_n(h).
\]
\end{lemma}

\begin{proof}
The first estimate follows immediately from Lemma \ref{lemma:linfty} (note that the phases of the $\mathbf{u}_{k,n}^{(i)}$ at $p_0$ can all be chosen equal at this point). The second one is a consequence of Lemma \ref{lemma:l2} combined with the fact that $\langle \mathbf{u}_{k,n}^{(i)},\mathbf{u}_{k,n}^{(j)}\rangle_{L^2} = \mc{O}_n(h^\infty)$ as these states have disjoint microsupport by \eqref{equation:faim}.
\end{proof}

In addition, the following holds:

\begin{lemma}
\label{lemma:pleurs2}
Let $b \in C^\infty_{\comp}(D^*\Sigma)$. Then:
\[
\langle \Op_k(b) \mathbf{u}_{k,n},\mathbf{u}_{k,n}\rangle_{L^2} = \dfrac{1}{T_E} \int_{0}^{2\pi} \int_0^{T_E} \chi_n(\theta)^2 b(t,\theta) \dd t \dd \theta + \mc{O}_n(h^{1/4}).
\]
\end{lemma}

\begin{proof}
Same proof as Lemma \ref{lemma:scl} using that the $\mathbf{u}_{k,n}^{(i)}$'s have disjoint microsupport.
\end{proof}

We can now complete the proof of Theorem \ref{theorem:defect}, item (ii).

\begin{proof}[Proof of Theorem \ref{theorem:defect}, item (ii)]
The proof is based on a standard diagonal argument that we briefly recall. Let $(a_\ell)_{\ell \geq 0} \subset C^\infty_{\comp}(T^*\Sigma)$ be a countable set which is dense in the space $C_{\comp}(T^*\Sigma)$ of compactly supported continuous functions (equipped with the sup norm). Fix $j \geq 0$. We may then choose $n := n(j) \geq 0$ large enough such that for all $0 \leq \ell \leq j$:
\[
\left|\int_{\T^2(p_0,E)} \chi_n(\theta)^2 a_\ell(t,\theta) \dd t \dd \theta - \int_{\T^2(p_0,E)} a_\ell(t,\theta) \dd t \dd \theta\right| \leq 1/j.
\]
Then, for all $k \geq k(j)$ large enough, using Lemmas \ref{lemma:pleurs} and \ref{lemma:pleurs2}, we obtain that for all $0 \leq \ell \leq j$:
\[
\dfrac{\|\mathbf{u}_{k,n}\|_{L^\infty}}{\|\mathbf{u}_{k,n}\|_{L^2}} \geq C k^{1/2}, \quad \left|\langle \Op_k(a_\ell) \mathbf{u}_{k,n},\mathbf{u}_{k,n}\rangle_{L^2}  - \dfrac{1}{T_E}\int_{\T^2(p_0,E)} a_\ell(t,\theta) \dd t \dd \theta\right| \leq C/j,
\]
and $\|\mathbf{u}_{k,n}\|^2_{L^2} = 2\pi + \mc{O}(1/j)$, for some uniform constant $C > 0$.  Therefore, the normalized sequence of eigenstates $u_{k(j)} := \mathbf{u}_{k(j),n(j)}/\|\mathbf{u}_{k(j),n(j)}\|_{L^2}$ saturates Hörmander's bound and, by construction, the semiclassical limit of $(u_{k(j)})_{j \geq 0}$ is the normalized Lebesgue measure on $\T^2(p_0,E)$.
\end{proof}}

\subsubsection{Absence of saturation elsewhere}

To complete our study of magnetic zonal states, we prove that magnetic zonal states only saturate the Hörmander bound at $p=p_0$:

\begin{lemma}
\label{lemma:not-saturated}
Let $\mathbf{u}_k$ be the eigenfunction defined in \eqref{equation:calor}. Let $p \in \Sigma, p \neq p_0$. Then, $|\mathbf{u}_{k}(p)| = o(h^{-1/2})$.
\end{lemma}

\begin{proof}
We use \cite[Theorem 4]{Galkowski-Toth-18} (this reference deals with the standard Laplace-Beltrami operator, but the proof extends to the magnetic Laplacian as well). Recall that $C(p,E)$ denotes the intersection of $T^*_p\Sigma$ with the energy shell $\{p=E\}$. For $p \in \Sigma$, and $T > 0$, we let
\[
\Lambda_{p,T} := \bigcup_{t \in [-T,T]} \Phi_t(C(p,E)).
\]

To apply \cite[Theorem 4]{Galkowski-Toth-18}, it suffices to show that there exists $T > 0$ such that the support of the defect measure $\mu$ associated with $(\mathbf{u}_k)_{k \geq 0}$ intersected with $\Lambda_{p,T}$ has Hausdorff dimension $< 2$. However, it follows from Theorem \ref{theorem:defect} and \eqref{equation:local-diffeo} that this intersection consists of (at most) a $1$-dimensional immersed submanifold, thus proving the claim.
\end{proof}

\subsection{Polynomial improvement}\label{improvsmallball} Finally, we establish Theorem \ref{theorem:criticalregime}.

\begin{proof}[Proof of Theorem \ref{theorem:criticalregime}] Let $x \in \Sigma$ and $\chi_\delta \in C^\infty(\Sigma)$ be a nonnegative bump function equal to $1$ on $B(x,C\delta)$ and $0$ outside of $B(x,2C\delta)$, where $C > 0$ is given by Proposition \ref{proposition:bound}. This function can be constructed such that
\begin{equation}
\label{equation:blowup}
\|\chi_\delta\|_{C^N(\Sigma)} \leq C_N \delta^{-N},
\end{equation}
for all $N \geq 0$.  We take $\delta = k^{-\eps}$, where $0 < \eps < 1/2$ is to be chosen later. Applying Proposition \ref{proposition:bound} with the integer $N=2026$, we find (denoting by $\dd x$ the Riemannian volume):
\[
|u_k(x)|^2\lesssim k^{1+\eps} \int_{D(x,C\delta)} |u_k(x)|^2 \dd x + k^{-2026} \lesssim k^{1+\eps} \langle \chi_\delta u_k,u_k \rangle_{L^2}+ k^{-2026}.
\]

Using \eqref{equation:concentration} and \eqref{equation:blowup}, we find
\[
\begin{split}
 \langle \chi_\delta u_k, u_k \rangle_{L^2(\Sigma, L^{\otimes k})} & = \int_{\Sigma} \chi_\delta(x) \dd x + \mc{O}(k^{-\theta \min(\ell,1/15)/4100}\|\chi_\delta\|_{C^{17}}) \\
 & = \mc{O}(\delta^2 + k^{-\theta \min(\ell,1/15)/4100}\delta^{-17}) = \mc{O}(k^{-2\eps} + k^{-\theta \min(\ell,1/15)/4100}k^{17 \eps}).
 \end{split}
 \]
 
Combining the previous estimates yields
\[
\|u_k\|_{L^\infty}^2 \lesssim k^{1-\eps} + k^{1+18\eps-\alpha} + k^{-2026},
\]
where $\alpha := \theta \min(\ell,1/15)/4100$. The optimal choice for $\eps$ is $\eps = \alpha/19$. Consequently, we find:
\[
\|u_k\|_{L^{\infty}} \lesssim k^{1/2- \theta \min(\ell,1/15)/155800}.
\]
This proves the claim.
\end{proof}

\bibliographystyle{alpha}
\bibliography{Biblio}

\end{document}